\documentclass[11pt, letterpaper, leqno]{article}
\usepackage{amsmath,amsthm,amscd,amssymb,amsbsy,cite,amstext,mathtools}
\usepackage{algorithm,algorithmic}
\usepackage{pgf}
\usepackage{color}
\usepackage{enumerate}
\usepackage{mathrsfs}
\usepackage[symbol]{footmisc}
\usepackage[colorlinks=true, pdfstartview=FitV, linkcolor=blue, citecolor=orange, urlcolor=blue]{hyperref}
\usepackage{tikz-cd}
\date{}
\usepackage{lineno}
\usepackage{hyperref}
\usepackage{amsfonts}
\usepackage[T1]{fontenc}
\usepackage{euler}

\theoremstyle{plain}
\newtheorem{theorem}{Theorem}
\newtheorem{lemma}{Lemma}[section]

\newtheorem{claim}{Claim}[section]
\newtheorem{problem}{Problem}

\theoremstyle{remark}
\newtheorem{remark}{Remark}[section]
\newtheorem*{acknowledgment}{Acknowledgment}

\theoremstyle{definition}

\newtheorem{definition}{Definition}[section]

\numberwithin{equation}{section}
\newcommand{\eqindent}{\displayindent0pt\displaywidth\textwidth}

\makeatletter
\newenvironment{subtheorem}[1]{%
	\def\subtheoremcounter{#1}%
	\refstepcounter{#1}%
	\protected@edef\theparentnumber{\csname the#1\endcsname}%
	\setcounter{parentnumber}{\value{#1}}%
	\setcounter{#1}{0}%
	\expandafter\def\csname the#1\endcsname{\theparentnumber.\Alph{#1}}%
	\expandafter\def\csname theH#1\endcsname{theorem.\theparentnumber.\Alph{#1}}%
	\unskip\ignorespaces
}{%
	\setcounter{\subtheoremcounter}{\value{parentnumber}}%
	\ignorespacesafterend
}
\makeatother
\newcounter{parentnumber}

\newcommand{\LA}[1]{\refstepcounter{equation}\text{(\theequation)}\label{#1}}
\newcommand{\LAQ}[2]{\begin{itemize}\item[\LA{#1}]{#2} \end{itemize}}
\numberwithin{equation}{section}

\newcommand{\supp}[1]{\mathrm{supp}\left( {#1} \right)}
\newcommand{\norm}[1]{\| {#1}\| }
\newcommand{\diam}{\mathrm{diam}\,}
\newcommand{\set}[1]{\left\{#1\right\}}
\newcommand{\void}{\varnothing}
\newcommand{\abs}[1]{\left|#1\right|}
\newcommand{\dist}[2]{\mathrm{dist}\left({#1}, {#2}\right)}

\newcommand{\jet}{\mathscr{J}}
\newcommand{\Q}{\mathcal{Q}}
\newcommand{\R}{\mathbb{R}}
\newcommand{\brac}[1]{\left(#1\right)}
\renewcommand{\d}{\partial}
\newcommand{\Rn}{\mathbb{R}^n}
\newcommand{\grad}{\nabla}
\newcommand{\qt}[1]{``\hspace{1sp}#1\,''}

\newcommand{\E}{\mathcal{E}}
\newcommand{\ct}{C^2}
\newcommand{\ctp}{C^2_+}
\newcommand{\rt}{\R^2}
\renewcommand{\P}{\mathcal{P}}
\newcommand{\pos}{[0,\infty)}

\newcommand{\for}{\enskip\text{for}\enskip}
\newcommand{\ddtm}{\frac{d^m}{dt^m}}
\newcommand{\dq}{\delta_Q}
\newcommand{\Ls}{\Lambda^\sharp}

\newcommand{\ul}[1]{\underline{#1}}
\newcommand{\A}{\ul{\mathcal{A}}}
\newcommand{\depth}{\mathrm{depth}}
\renewcommand{\phi}{\varphi}
\newcommand{\vp}{\vec{P}}

\begin{document}

	\title{Algorithms for Nonnegative $C^2(\mathbb{R}^2)$ Interpolation}
	\author{Fushuai Jiang \and Garving~K. Luli}
	\date{\today}
	\maketitle

	\begin{abstract}
		Let $ E \subset \mathbb{R}^2 $ be a finite set, and let $ f : E \to [0,\infty) $. In this paper, we address the algorithmic aspects of nonnegative $C^2$ interpolation in the plane. Specifically, we provide an efficient algorithm to compute a nonnegative $C^2(\mathbb{R}^2)$ extension of $ f $ with norm within a universal constant factor of the least possible. We also provide an efficient algorithm to approximate the trace norm. 
	\end{abstract}

	\section{Introduction}

	For integers $ m\geq 0,n \geq 1 $, we write $ C^m(\Rn) $ to denote the Banach space of $ m $-times continuously differentiable real-valued functions such that the following norm is finite
	\begin{equation*}
		\norm{F}_{C^m(\Rn)} := \sup\limits_{x \in \Rn}\max_{\abs{\alpha} \leq m}\abs{ \partial^\alpha F(x) }.
	\end{equation*}
	We write $ C^m_+(\Rn) $ to denote the collection of nonnegative functions in $ C^m(\Rn) $. 
	Let $ E \subset \Rn $ be finite. 
	We write $ C^m_+(E) $ to denote the collection of functions $ f: E \to \pos $. If $ S $ is a finite set, we write $ \#(S) $ to denote the number of elements in $ S $. We use $ C $ to denote constants that depend only on $ m $ and $ n $.

	In this paper, we provide algorithmic solutions to the following problems for $ m = n = 2 $. These algorithms were announced in \cite{JL20,JL20-Ext}. 
	
	\begin{problem}\label{prob.norm}
		Let $ E \subset \Rn $ be a finite set. Let $ f : E \to \pos $. Compute the order of magnitude of 
		\begin{equation}
			\norm{f}_{C^m_+(E)} := \inf \set{ \norm{F}_{C^m(\Rn)} : F|_E = f\text{ and } F \geq 0 }\,. \label{eq.trace}
		\end{equation}
	\end{problem}
	
	\begin{problem}\label{prob.interpolant}
		Let $ E \subset \Rn $ be a finite set. Let $ f : E \to \pos $. Compute a nonnegative function $ F \in C^m(\Rn) $ such that $ F|_E = f $ and $ \norm{F}_{C^m(\Rn)} \leq C\norm{f}_{C^m_+(E)} $. 
	\end{problem}

	By \qt{order of magnitude} we mean the following: Two quantities $ M $ and $ \tilde{M} $ determined by $ E,f,m,n $ are said to have the \underline{same order of magnitude} provided that $ C^{-1}M \leq \tilde{M} \leq CM $, with $ C $ depending only on $ m $ and $ n $. To compute the order of magnitude of $ \tilde{M} $ is to compute a number $ M $ such that $ M $ and $ \tilde{M} $ have the same order of magnitude. 
	
	By \qt{computing a function $F$} from $ (E,f) $, we mean the following: After processing the input $ (E,f) $, we are able to accept queries consisting of a point $ x \in \Rn $, and produce a list of numbers $ (f_\alpha(x) : \abs{\alpha} \leq m) $. The algorithm \qt{computes the function $ F $} if for each $ x \in \Rn $, we have $ \d^\alpha F(x) =f_\alpha(x) $ for $ \abs{\alpha} \leq m $. 
	
	Problem \ref{prob.interpolant} is an open problem posed in \cite{FI20-book}, and Problem \ref{prob.norm} is closely related to Problem \ref{prob.interpolant}. The theoretical aspects of the problems for $ m = n = 2 $ were addressed in \cite{JL20,JL20-Ext}. We refer the readers to \cite{JL20,JL20-Ext} for a more thorough discussion on the problems. 
	
	In this paper, we content ourselves with an idealized computer with standard von Neumann architecture that is able to process exact real numbers. We refer the readers to \cite{FK09-Data-2} for discussion on finite-precision computing.

	In \cite{JL20-Ext}, we proved the following.
	
	\begin{theorem}\label{thm.bd}
		Let $ E \subset \rt $ be a finite set. There exist (universal) constants $C, D$, and a map $ \E :   \ctp(E) \times \pos \to \ctp(\rt)  $ such that the following hold. 
		
		\begin{enumerate}[(A)]
			\item Let $ M \geq 0 $. Then for all $ f \in \ctp(E) $ with $ \norm{f}_{\ctp(E)} \leq M $, we have $ \E(f,M) = f $ on $ E $ and 
			$ \norm{\E(f,M)}_{\ct(\rt)} \leq CM $.
			
			\item For each $ x \in \rt $, there exists a set $ S(x) \subset E $ with $ \# (S(x)) \leq D $ such that for all $ M \geq 0 $ and $ f, g \in \ctp(E) $ with $ \norm{f}_{\ctp(E)}, \norm{g}_{\ctp(E)} \leq M $ and $ f|_{S(x)} = g|_{S(x)} $, we have
			\begin{equation*}
				\d^\alpha \E(f,M)(x) = \d^\alpha \E(g,M)(x)
				\for
				\abs{\alpha} \leq 2\,.
			\end{equation*}
			
		\end{enumerate}
		
	\end{theorem}
	
	A few remarks on Theorem \ref{thm.bd} are in order. First of all, in \cite{JL20}, we showed that the extension operator $\E$ cannot be linear in general. The constant $D$ appearing in Theorem \ref{thm.bd} is called the \underline{depth} of the extension operator $\E$. This generalizes the notion of the depth of a {\em linear} extension operator first studied by C. Fefferman in \cite{F05-L,F07-St} (for further discussion on the depth of linear extension operators see also G.K. Luli \cite{L10}). The depth of an extension operator (both linear and nonlinear) measures the computational complexity of the extension. The existence of a linear extension operator of bounded depth is one of the main ingredients for the Fefferman-Klartag \cite{FK09-Data-1, FK09-Data-2}  and Fefferman  \cite{F09-Data-3} algorithms for solving the interpolation problems without the nonnegative constraints; the algorithms in \cite{F09-Data-3,FK09-Data-1, FK09-Data-2} are likely essentially the best possible.

	In this paper, we will provide another proof of Theorem \ref{thm.bd} but with algorithmic complexity in mind. This is the content of Theorem \ref{thm.bd-alg}. 
	
	We start with a definition.
	
	\begin{definition}\label{def.depth}
		\newcommand{\bn}{{\bar{N}}}
		\newcommand{\rbn}{\R^\bn}
		Let $ \bn \geq 1 $ be an integer. Let $ \mathcal{B} = \set{\xi_1, \cdots, \xi_\bn} $ be a basis of $ \rbn $. Let $ \Omega \subset \rbn $ be a subset. Let $ X $ be a set. Let $ \Xi: \Omega \to X $ be a map.
		\begin{itemize}
			\item We say $ \Xi $ \underline{has depth $D$}, if there exists a $ D $-dimensional subspace $ V = {\rm span}\brac{\xi_{i_1}, \cdots, \xi_{i_D}} $, $ \xi_{i_1}, \cdots, \xi_{i_D} \in \mathcal{B} $, such that for all $ z_1, z_2 \in\Omega $ with $ \pi_V(z_1) = \pi_V(z_2) $, we have $ \Xi(z_1) = \Xi(z_2) $. Here, $ \pi_V : \rbn \to V $ is the natural projection. 
			\item Suppose $ \Xi $ has depth $ D $. Let $  V = {\rm span}\brac{\xi_{i_1}, \cdots, \xi_{i_D}} $ be as above. By an \underline{efficient representation} of $ \Xi $, we mean a specification of the index set $ \set{i_1, \cdots, i_D} \subset \set{1, \cdots, \bn} $ and an algorithm to compute a map $\tilde{\Xi}:\Omega\cap V \to X$ in $ C_D $ operations, i.e., given an input $\omega \in \Omega \cap V$, we can compute $\tilde{\Xi}(\omega)$ in $C_D$ operations. Here, the map $ \tilde{\Xi} $ agrees with $ \Xi $ on $ \Omega\cap V $, and $C_D$ is a constant depending only on $D$.

		\end{itemize}

	\end{definition}

	Note that in general, the set $ \Omega $ may have complicated geometry. For the purpose of this paper, we will only be considering when $ \Omega $ is some Euclidean space or the first quadrant of some Euclidean space.

	\begin{remark}\label{rem.depth}
		
		Suppose $\Xi:\R^{\bar{N}}\to \R$ is a linear functional. Recall from \cite{FK09-Data-2} that a \qt{compact representation} of a linear functional $\Xi:\R^{\bar{N}}\to \R$ consists of a list of indices $ \set{i_1, \cdots, i_D} \subset \set{1, \cdots, \bar{N}} $ and a list of coefficients $ \chi_{i_1}, \cdots, \chi_{i_D} $, so that the action of $\Xi$ is characterized by
		\begin{equation*}
			\Xi : (\xi_1, \cdots, \xi_{\bar{N}}) \mapsto \sum_{\Delta = 1}^{D}\chi_{i_\Delta} \cdot \xi_{i_\Delta}.
		\end{equation*}
		Therefore, given $v \in \mathrm{span}(\xi_{i_1},\cdots,\xi_{i_D})$, we can compute $\Xi(v)$ by computing the dot product of two vectors of length $D$, which requires $C_D$ operations. The present notion of \qt{efficient representation} is a natural generalization adapted to the nonlinear nature of nonnegative interpolation (see \cite{JL20,JL20-Ext}). Since a nonlinear map in general does not admit a simple representation, we emphasize the complexity of an extension operator rather than its structure.

	\end{remark}

	We think of $ \ctp(E) \cong \pos^N $. We use the standard orthonormal frame $ \R^{N} $ as a basis for the purpose of defining finite depth. We write $ \P^+ $ to denote the vector space of polynomials with degree no greater than two, and we write $ \jet_x^+ F $ to denote the two-jet of $ F $ at $ x $.
	
	The main theorem of the paper is the following.
	
	\begin{theorem}\label{thm.bd-alg}
		Suppose we are given a finite set $ E \subset \rt $ with $ \#(E) = N $. Then there exists a collection of maps $ \set{\Xi_x : x \in \rt} $, where $ \Xi_x : \ctp(E) \times [0,\infty) \to \P^+ $ for each $ x \in \rt $, such that the following hold.
		\begin{enumerate}[(A)]
			\item There exists a universal constant $ D $ such that for each $ x \in \rt $, the map $ \Xi_x(\cdot\,,\cdot) : \ctp(E)\times 
			\pos \to \P^+ $ is of depth $ D $.

			\item Suppose we are given $ (f,M) \in \ctp(E) \times \pos $ with $ \norm{f}_{\ctp(E)} \leq M $.
			Then there exists a function $ F \in \ctp(\rt) $ such that
			\begin{equation*}
				\begin{split}
					\jet^+_x F = \Xi_x(f,M) \text{ for all } x \in \rt,\,
					\norm{F}_{\ct(\rt)} \leq CM,
					\text{ and }
					F(x) = f(x)
					\for
					x \in E.
				\end{split}
			\end{equation*}

			\item There is an algorithm, that takes the given data set $E$, performs one-time work, and then responds to queries.
			
			A query consists of a point $ x \in \rt $, and the response to the query is the depth-$ D $ map $ \Xi_x $, given in its efficient representation (see Definition \ref{def.depth}).
			
			The one-time work takes $ CN\log N $ operations and $ CN $ storage. The work to answer a query is $ C\log N $. 
		\end{enumerate}
	\end{theorem}

	\begin{remark}\label{rem.Sx}
		Theorem \ref{thm.bd-alg}(C) implies that for each $ x \in \rt $, there exists a set $ S(x) \subset E $ with $ \#(S(x)) \leq D $ such that for all $ (f,M), (g,M) \in \ctp(E)\times\pos $ with $ \norm{f}_{\ctp(E)}, \norm{g}_{\ctp(E)} \leq M $ and $ f|_{S(x)} = g|_{S(x)} $, we have $ \Xi_x(f,M) = \Xi_x(g,M) $. Moreover, after one-time work using at most $ CN\log N $ operations and $ CN $ storage, we can perform the following task: Given $ x \in \rt $, we can produce the set $ S(x) $ using no more than $ C\log N $ operations. 
	\end{remark}

	Using Theorem \ref{thm.bd-alg}, we obtain an algorithmic version of the Sharp Finiteness Principle (see {Theorem~5} in \cite{JL20}):

	\begin{theorem}[Algorithmic Sharp Finiteness Principle]\label{thm.sfp}
		Let $ E \subset \R^2 $ with $ \#(E) = N < \infty $. Then there exist universal constants $ C_1, C_2 ,C_3, C_4, C_5 $ and a list of subsets $ S_1, S_2, \cdots, S_L \subset E $ satisfying the following.
		\begin{enumerate}[(A)]
			\item We can compute the list $ \set{S_\ell: \ell = 1, \cdots, L} $ from $E, $ using one-time work of at most $ C_1N\log N $ operations, and using storage at most $ C_2N $. 
			\item $ \#(S_\ell) \leq C_3  $ for each $ \ell = 1, \cdots, L $.
			\item $ L \leq C_4N $.
			\item Given any $ f : E \to [0,\infty) $, we have
			\begin{equation*}
				\max_{\ell = 1, \cdots, L}\norm{f}_{C^2_+(S_\ell)} \leq \norm{f}_{C^2_+(E)} \leq C_5 \max_{\ell = 1, \cdots, L}\norm{f}_{C^2_+(S_\ell)}\,.
			\end{equation*}
		\end{enumerate}
		
	\end{theorem}
	
	Theorem \ref{thm.sfp} without condition (A) is the same as Theorem 5 in \cite{JL20}.
	
	In this paper, we will prove Theorem \ref{thm.sfp} via Theorem \ref{thm.bd-alg}. Our approach yields an alternate proof of Theorem 5 in \cite{JL20}. The list of subsets $ \set{S_\ell: \ell = 1, \cdots, L} $ that arises in this paper may be different from that in Theorem 5 of \cite{JL20}. It will be interesting to understand the relationship between them. 
	
	Using Theorem \ref{thm.sfp}, we can produce Algorithm \ref{alg.norm}, solving Problem \ref{prob.norm}.

	\begin{algorithm}[H]
		\caption{Nonnegative $ C^2(\R^2) $ Interpolation Algorithm - Trace Norm}\label{alg.norm}
		
		\begin{itemize}
			\item[] \textbf{DATA:} $ E \subset \R^2 $ finite with $ \#(E) = N $.
			\item[] \textbf{QUERY:} $f: E \to 
			\pos $.
			\item[] \textbf{RESULT}: The order of magnitude of $ \norm{f}_{C^2_+(E)} $. More precisely, the algorithm outputs a number $ M \geq 0 $ such that both of the following hold.
			\begin{itemize}
				\item We guarantee the existence of a function $ F \in C^2_+(\R^2) $ such that $ F|_E = f $ and $ \norm{F}_{C^2(\R^2)} \leq CM $.
				\item We guarantee there exists no $ F \in C^2_+(\R^2) $ with norm at most $ M $ satisfying $ F|_E = f $.
			\end{itemize}
			
			\item[] \textbf{COMPLEXITY:} 
			\begin{itemize}
				\item Preprocessing $ E $: at most $ CN\log N $ operations and $ CN $ storage.
				\item Answer query: at most $ CN $ operations.
			\end{itemize} 
		\end{itemize}
		
	\end{algorithm}

	Using Theorem \ref{thm.bd-alg}, we can produce Algorithm \ref{alg.interpolant}, solving Problem \ref{prob.interpolant}.
	
	\begin{algorithm}[H]
		\caption{Nonnegative $ C^2(\R^2) $ Interpolation Algorithm - Interpolant}\label{alg.interpolant}
		\begin{itemize}
			\item[] \textbf{DATA:} $ E \subset \R^2 $ finite with $ \#(E) = N $. $ f: E \to \pos $. $ M \geq 0 $. 
			\item[] \textbf{ORACLE:} $ \norm{f}_{C^2_+(E)} \leq M $. 
			
			\item[] \textbf{RESULT}: A query function that accepts a point $ x \in \mathbb{R}^2 $ and produces a list of numbers $ ( f_\alpha(x) : \abs{\alpha} \leq 2) $ that guarantees the following: There exists a function $ F \in C^2_+(\R^2) $ with $ \norm{F}_{C^2(\R^2)} \leq CM$ and $ {F}|_E = f $, such that $ \d^\alpha F(x) = f_\alpha(x) $ for $ \abs{\alpha} \leq 2 $. The function $ F $ is independent of the query point $ x $, and is uniquely determined by $ (E,f,M) $.
			
			\item[] \textbf{COMPLEXITY:} 
			
			\begin{itemize}
				\item Preprocessing $ E $: at most $ CN\log N $ operations and $ CN $ storage.
				
				\item Answer query: at most $ C\log N $ operations.
			\end{itemize}
			
		\end{itemize}

	\end{algorithm}

	Theorem \ref{thm.bd-alg} also yields Algorithm \ref{alg.dependence} for computing the representative sets $S_{\ell}$ in Theorem \ref{thm.sfp}.

	\begin{algorithm}[H]
		\caption{Nonnegative $ C^2(\R^2) $ Interpolation Algorithm - Representative Sets}\label{alg.dependence}
		\begin{itemize}
			\item[] \textbf{DATA:} $ E \subset \R^2 $ finite with $ \#(E) = N $. 
			
			\item[] \textbf{RESULT}: A query (set-valued) function that accepts a point $ x \in \mathbb{R}^2 $ and produces a subset $ S(x) \subset E $, where $ S(x) $ agrees with that in Remark \ref{rem.Sx}. 
			
			\item[] \textbf{COMPLEXITY:} 
			
			\begin{itemize}
				\item Preprocessing $ E $: at most $ CN\log N $ operations and $ CN $ storage.
				
				\item Answer query: at most $ C\log N $ operations.
			\end{itemize}
			
		\end{itemize}

	\end{algorithm}

	To see how to produce Algorithm \ref{alg.dependence} from Theorem \ref{thm.bd-alg}, we simply note that each map $ \Xi_x $ in Theorem \ref{thm.bd-alg} is stored in its efficient representation (see Definition \ref{def.depth}). Thus, the set $ S(x) $ is given by the corresponding set of indices in the efficient representation of $ \Xi_x $.

	\begin{acknowledgment}
		We are indebted to 
		Jes\'us A. De Loera, Charles Fefferman, Kevin O'Neill, Naoki Saito, and Pavel Shvartsman, for their valuable comments. We also thank all the participants in the 11th Whitney workshop for fruitful discussions, and Trinity College Dublin for hosting the workshop. 
		
		This project is supported by NSF Grant DMS-1554733. The first author is supported by the UC Davis Summer Graduate Student Researcher Award and the Alice Leung Scholarship in Mathematics. The second author is supported by the UC Davis Chancellor's Fellowship. 
	\end{acknowledgment}

	\section{Preliminaries}

	\newcommand{\touch}{\leftrightarrow}
	
	We use $ c_*, C_*, C' $, etc., to denote universal constants. They may be different quantities in different occurrences. We will label them to avoid confusion when necessary.
	
	We assume that we are given an ordered orthogonal coordinate system on $ \rt $, specified by a pair of unit vectors $ [e_1,e_2] $. We use $ \abs{\,\cdot\,} $ to denote Euclidean distance. We use $ B(x,r) $ to denote the disk of radius $ r $ centered at $ x $. For $ X, Y \subset \rt $, we write $ \dist{X}{Y} := \inf_{x \in X, y \in Y}\abs{x - y} $. 
	
	We use $ \alpha = (\alpha_1, \alpha_2),\beta = (\beta_1, \beta_2) \in \mathbb{N}_0^2 $, etc., to denote multi-indices. We write $ \d^\alpha $ to denote $ \d_{e_1}^{\alpha_1}\d_{e_2}^{\alpha_2} $. We adopt the partial ordering $ \alpha \leq \beta $ if and only if $ \alpha_i \leq \beta_i $ for $ i = 1,2 $. 
	
	By a square, we mean a set of the form $ Q = [a,a+\delta) \times [b,b+\delta) $ for some $ a,b \in \R $ and $ \delta > 0 $. If $ Q $ is a square, we write $ \dq $ to denote the sidelength of the square. For $ \lambda > 0 $, we use $ \lambda Q $ to denote the square whose center is that of $ Q $ and whose sidelength is $ \lambda\dq $. Given two squares $ Q, Q' $, we write $ Q\touch Q' $ if $ closure(Q) \cap closure(Q') \neq \void $.
	
	A dyadic square is a square of the form $ Q = [2^k \cdot i, 2^k\cdot (i + 1) ) \times [2^k\cdot j, 2^k\cdot (j + 1)) $ for some $ i,j,k \in \mathbb{Z} $. Each dyadic square $ Q $ is contained in a unique dyadic square with sidelength $ 2\dq $, denoted by $ Q^+ $.

	Let $ \Omega \subset \Rn $ be a set with nonempty interior $ \Omega_0 $ such that $ \Omega \subset \overline{\Omega}_0 $. For nonnegative integers $ m,n $, we use $ C^m(\Omega) $ to denote the vector space of $ m $-times continuously differentiable real-valued functions up to the closure of $ \Omega $, whose derivatives up to order $ m $ are bounded. For $ F \in C^m(\Omega) $, we define
	\begin{equation*}
		\norm{F}_{C^m(\Omega)} := \sup_{x \in \Omega_0} \max_{\abs{\alpha}\leq m}\abs{\d^\alpha F(x)}.
	\end{equation*}
	We write $ C^m_+(\Omega) $ to denote the collection of functions $ F \in C^m(\Omega) $ such that $ F \geq 0 $ on $ \Omega $. 
	
	Let $ E \subset \Rn $ be finite. We define the following.
	\begin{equation*}
		\begin{split}
			C^m(E) := \set{f:E \to \R} \cong \R^{\#(E)}
			\enskip&\text{ and }\enskip
			\norm{f}_{C^m(E)}  := \inf\set{\norm{F}_{C^m(\Rn)} : F |_E = f};\\
			C^m_+(E) := \set{f:E \to \pos} \cong \pos^{\#(E)}
			\enskip&\text{ and }\enskip
			\norm{f}_{C^m_+(E)} := \inf\set{\norm{F}_{C^m(\Rn)} : F|_E = f
				\text{ and }F\geq 0}\,.
		\end{split}
	\end{equation*}

	\subsection{Polynomials and Whitney fields}
	
	\newcommand{\ring}{\mathcal{R}}

	We write $ \P $ to denote the vector space of affine polynomials on $ \rt $. It is a three-dimensional vector space. We use $ \P^+ $ to denote the vector space of polynomials in $ \rt $ with degree no greater than two. It is a six-dimensional vector space.

	For $ x \in \R^2 $ and a function $ F $ twice continuously differentiable at $ x $, we write $ \jet_x F$, $\jet_x^+F $ to denote the one-jet, two-jet of $ F $ at $ x $, respectively, which we identify with the degree-one, degree-two Taylor polynomials, respectively, 
	\begin{equation}
		\begin{split}
			\jet_x F(y) &:= \sum_{\abs{\alpha} \leq 1}\frac{\d^\alpha F(x)}{\alpha !}(y-x)^\alpha, \text{ and }\\
			\jet_x^+ F(y) &:= \sum_{\abs{\alpha} \leq 2}\frac{\d^\alpha F(x)}{\alpha !}(y-x)^\alpha.
		\end{split}
		\label{eq.jet-def}
	\end{equation}
	We use $ \ring_x$, $\ring^+_x $ to denote the rings of one-jets, two-jets at $ x $, respectively. The multiplications on $ \ring_x $ and $ \ring_x^+ $ are defined in the following way:
	\begin{equation*}
		P\odot_x R := \jet_x(PR)
		\text{ and }
		P^+\odot_x^+R^+:= \jet_x^+(P^+R^+),
	\end{equation*}
	for $ P,R \in \ring_x $ and $ P^+, R^+ \in \ring_x^+ $.

	\newcommand{\wt}{W^2}
	\newcommand{\wtp}{W^2_+}
	Let $ S \subset \Rn $ be a nonempty finite set. A \underline{Whitney field} on $ S $ is an array of polynomials
	\begin{equation*}
		\vec{P} := (P^x)_{x \in S},
		\enskip\text{ where }
		\enskip 
		P^x \in \ring_x
		\text{ for each } x \in S\,.
	\end{equation*}
	Given $ \vec{P} = (P^x)_{x \in S} $, we sometimes use the notation
	\begin{equation*}
		(\vec{P},x) := P^x
		\for x \in S\,.
	\end{equation*}
	We write $ \wt(S) $ to denote the vector space of all Whitney fields on $ S $. For $ \vec{P} = (P^x)_{s \in S} \in \wt(S) $, we define
	\begin{equation*}
		\norm{\vec{P}}_{\wt(S)} := \max_{x \in S, \abs{\alpha} \leq 1}\abs{\d^\alpha P^x(x)} +  \max_{\substack{x, y \in S,\,x \neq y,\,\abs{\alpha}\leq 1}}\frac{\abs{\d^\alpha(P^x - P^y)(x)}}{\abs{x - y}^{2-\abs{\alpha}}}.
	\end{equation*}
	We note that $ \norm{\cdot}_{\wt(S)} $ is a norm on $ \wt(S) $. 
	
	We write $ \wtp(S) $ to denote a subcollection of $ \wt(S) $, such that $ \vec{P} \in W^2_+(S) $ if and only if for each $ x \in S $, there exists some $ M_x \geq 0 $ such that 
	\begin{equation}
		(\vec{P},x)(y) + M_x\abs{y - x}^2 \geq 0
		\text{ for all } y \in \rt\,.
		\label{eq.MP}
	\end{equation}
	For $ \vec{P} \in \wtp(S) $, we define
	\begin{equation*}
		\norm{\vec{P}}_{\wtp(S)} := \norm{\vec{P}}_{\wt(S)} + \max_{x \in S} \brac{\inf \set{ M_x \geq 0 : (\vec{P},x)(y) + M_x\abs{y-x}^2 \geq 0 \, \text{ for all } y \in \rt }}.
	\end{equation*}
	
	The next lemma is a Taylor-Whitney correspondence for $ \ctp(\rt) $. (A) is simply Taylor's theorem. See \cite{JL20,FIL16+} for a proof of (B). 
	
	\begin{lemma}\label{lem.WT}
		There exists a universal constant $ C_w $ such that the following holds.
		
		Let $ E \subset \rt $ be a finite set. 
		
		\begin{enumerate}[(A)]
			\item Let $ F \in \ctp(\rt) $. Let $ \vec{P} := (\jet_x F)_{x \in E} $. Then $ \vec{P} \in \wtp(E) $
			and
			$ \norm{\vec{P}}_{\wtp(E)} \leq C_w\norm{F}_{\ct(\rt)} $. 
			\item There exists a map $ T_w^E : \wtp(E) \to \ctp(\rt) $ such that $ \norm{T_w^E(\vec{P})}_{\ct(\rt)} \leq C_w\norm{\vec{P}}_{\wtp(E)} $ and $ \jet_x T_w^E(\vec{P}) = (\vec{P},x) $ for each $ x \in E $. 
		\end{enumerate}
	\end{lemma}
	
	\subsection{Trace norm on small subsets}
	\label{sect:lasso}
	
	Let $ S \subset \rt $ be a finite set. We define the following two functions.
	
	\begin{equation}
		\begin{split}
			\Q = \Q_S : \wt(S) &\to \pos\\
			\vec{P} = (P^x)_{x \in S} &\mapsto \sum_{\substack{x \in S\\\abs{\alpha}\leq 1}}\abs{\d^\alpha P^x(x)} + \sum_{\substack{x, y \in S\\x \neq y\\\abs{\alpha}\leq 1}} { \frac{\abs{\d^\alpha (P^x - P^y)(x)}}{\abs{x - y}^{2-\abs{\alpha}}} }\,,
		\end{split}
		\label{eq.Q-def}
	\end{equation}
	and
	
	\newcommand{\M}{\mathcal{M}}
	
	\begin{equation}
		\begin{split}
			\M = \M_S : \wt(S) &\to [0,\infty]\\
			\vec{P} = (P^x)_{x \in S} &\mapsto \begin{cases}
				\sum\limits_{x \in S} \frac{\abs{\grad P^x}^2}{P^x(x)} &\text{ if $ P^x(x) \geq 0 $ for each $ x \in S $}\\
				\infty &\text{ if there exists $ x \in S $ such that $ P^x(x) < 0 $}
			\end{cases}\,.
		\end{split}
		\label{eq.M-def}
	\end{equation}
	
	In \eqref{eq.M-def}, we use the conventions that $ \frac{0}{0} = 0 $ and $ \frac{a}{0} = \infty $ for $ a > 0 $.

	\begin{lemma}\label{lem.Q+M}
		Let $ S \subset \rt $ be a finite set with $ \#(S) \leq N_0 $ for some universal constant $ N_0 $. Let $ \Q $ and $ \M $ be as in \eqref{eq.Q-def} and \eqref{eq.M-def}. Then there exists a universal constant $ C $ such that
		\begin{equation}
			C^{-1}\norm{\vec{P}}_{\wtp(S)} \leq ({\Q} + \M)(\vec{P}) \leq C\norm{\vec{P}}_{\wtp(S)}
			\text{ for all }
			\vec{P} \in \wtp(S).
			\label{eq.lem.3.1}
		\end{equation}
		Moreover, $ \vec{P} \in \wt(S)\setminus\wtp(S) $ if and only if $ \M(\vec{P}) = \infty $.

	\end{lemma}

	\begin{proof}
		
		We write $ C, C' $, etc., to denote universal constants. 
		
		Fix $ \vec{P} = (P^x)_{x \in S} \in \wtp(S) $.

		Suppose $ ({\Q}+\M)(\vec{P}) \leq M $. We want to show that
		\begin{equation}
			\norm{\vec{P}}_{\wtp(S)} \leq CM.
			\label{eq.2.2.6}
		\end{equation}

		Since each summand in the definition of $ \Q $ in \eqref{eq.Q-def} is nonnegative, we have
		\begin{equation}
			\max\limits_{x \in S,\abs{\alpha} \leq 1}\abs{\d^\alpha P^x(x)} \leq CM, \text{ and }
			\max\limits_{x,y \in S, x \neq y, \abs{\alpha} \leq 1}\frac{\abs{\d^\alpha(P^x - P^y)(x)}}{\abs{x - y}^{2-\abs{\alpha}}} \leq CM.
			\label{eq.2.2.4}
		\end{equation}
		
		Since $ \M(\vec{P}) \leq M $, we have
		\begin{equation}
			\abs{\nabla P^x}^2 \leq MP^x(x)
			\for x \in S.
			\label{eq.2.2.7}
		\end{equation}
		Therefore, we have
		\begin{equation}\label{eq.2.2.5}
			P^x(y) + \frac{M}{4}\abs{y-x}^{2} \geq 0
			\text{ for all } y \in \rt, x \in S.
		\end{equation}
		By the definition of $ \norm{\cdot}_{\wtp(S)} $, we see that \eqref{eq.2.2.6} follows from \eqref{eq.2.2.4} and \eqref{eq.2.2.5}.

		Suppose $ \norm{\vec{P}}_{\wtp(S)} \leq M $. We want to show that 
		\begin{equation}
			({\Q}+\M)(\vec{P}) \leq CM.
			\label{eq.2.2.1}
		\end{equation}
		
		By the definition of $ \norm{\cdot}_{\wtp(S)} $, we know that
		\begin{equation}\label{eq.2.6}
			\max\limits_{x \in S,\abs{\alpha} \leq 1}\abs{\d^\alpha P^x(x)} \leq M,\quad
			\max\limits_{x,y \in S, x \neq y, \abs{\alpha} \leq 1}\frac{\abs{\d^\alpha(P^x - P^y)(x)}}{\abs{x - y}^{2-\abs{\alpha}}}\leq M, \text{ and }
		\end{equation}
		\begin{equation}\label{eq.2.7}
			P^x(y) + M\abs{y-x}^{2} \geq 0
			\text{ for all } y \in \rt, x \in S.
		\end{equation}
		It follows from \eqref{eq.2.6} that 
		\begin{equation}
			\Q(\vec{P}) \leq CM^2.
			\label{eq.2.2.2}
		\end{equation}
		For each $ x \in S $, restricting $ P^x $ to each line in $ \rt $ passing through $ x $ and computing the discriminant, we can conclude from \eqref{eq.2.7} that
		\begin{equation}
			\abs{\nabla P^x}^2 \leq C M P^x(x)
			\for x \in S.
			\label{eq.2.8}
		\end{equation}
		It follows from \eqref{eq.2.8} that 
		\begin{equation}
			\M(\vec{P}) \leq CM.
			\label{eq.2.2.3}
		\end{equation}
		(Recall that we use the convention $ \frac{0}{0} = 0 $). \eqref{eq.2.2.1} then follows from \eqref{eq.2.2.2} and \eqref{eq.2.2.3}.

		This proves \eqref{eq.lem.3.1}.
		
		Now we turn to the second statement. 
		
		Suppose $ \vec{P} \in \wt(S) $ is such that $ \M(\vec{P}) = \infty $. Then at least one of the following holds:
		\begin{itemize}
			\item $ P^x(x) < 0 $ for some $ x \in S $, in which case condition \eqref{eq.MP} fails for such $ P^x $, so $ \vec{P} \notin \wtp(S) $.
			\item There exists $ x \in S $ such that $ P^x(x) = 0 $ but $ \nabla P^x \neq 0 $, in which case condition \eqref{eq.MP} fails for such $ P^x $, so $ \vec{P} \notin \wtp(S) $. 
		\end{itemize}
		In conclusion, we have $ \vec{P} \notin \wtp(S) $.
		
		Conversely, suppose $ \vec{P} \notin \wtp(S) $. Then there exists $ x \in S $ such that condition \eqref{eq.MP} fails for $ P^x $. This means that either $ P^x(x) < 0 $, or $ P^x(x) = 0 $ but $ \nabla P^x \neq 0 $. In either case, we have $ \M(\vec{P}) = \infty $. 
		
		Lemma \ref{lem.Q+M} is proved.

	\end{proof}

	\newcommand{\rnz}{{\R^{d}}}
	\newcommand{\wts}{\wt(S)}
	\newcommand{\Af}{{\mathbb{A}_f}}
	
	For the rest of the subsection, we fix a finite set $ S \subset \rt $ with $ \#(S) \leq N_0 $, where $ N_0 $ is a universal constant. We also fix a function $ f : S \to \pos $. We explain how to compute the order of magnitude of $ \norm{f}_{\ctp(S)} $.
	
	We adopt the following notation: For $ A , B \geq 0 $, we write $ A \approx B $ if there exists a universal constant $ C $ such that $ C^{-1}A \leq B \leq CA $.

	We define an affine subspace $ \Af \subset \wts $ by
	\begin{equation*}
		\begin{split}
			\Af &:= \set{\vec{P}= (P^x)_{x \in S} \in \wts : P^x(x) = f(x)
				\text{ and }
				f(x) = 0 \Rightarrow
				\nabla P^x = 0
				\for x \in S}\\
			&= \set{\vec{P}= (P^x)_{x \in S} \in \wtp(S) : P^x(x) = f(x) \for x \in S}.
		\end{split}
	\end{equation*}
	Note that $ \Af $ has dimension $ 2\cdot(\#(S)-\#(f^{-1}(0)) $.
	
	Let $ \Q $ and $ \M $ be as in \eqref{eq.Q-def} and \eqref{eq.M-def}. By Lemma \ref{lem.WT} and Lemma \ref{lem.Q+M}, 
	\begin{equation}
		\norm{f}_{\ctp(S)} \approx \inf \set{(\Q+\M)(\vec{P}) : \vec{P} \in \Af}.
		\label{eq.3.2.1}
	\end{equation}

	Let $ d := \dim W^2(S) = \#(S)\cdot\dim\P = 3\#(S) $. We identify $ W^2(S) \cong \R^d $ via $ (P^x)_{x \in S} \mapsto \brac{P^x(x), \d_{e_1}P^x, \d_{e_2}P^x}_{x \in S} $. We define the $ \ell^1 $ and $ \ell^2 $-norms, respectively, on $ \rnz $ by the formulae
	\begin{equation*}
	\norm{v}_{\ell^1}:= \sum_{i = 1}^{d}\abs{v_i}
	\text{ and }
	\norm{v}_{\ell^2}:= \brac{\sum_{i = 1}^{d}\abs{v_i}^2}^{1/2}
	\quad\for\quad v = (v_1, \cdots, v_{d}) \in \rnz.
	\end{equation*}

	Consider the following objects.

	\begin{itemize}

		\item Let $ L_w : W^2(S) \to \R^d $ be a linear isomorphism that maps $ \vec{P} \in \wts $ to the vector in $ \R^{d} $ with components
		\begin{equation*}
			\frac{\d^\alpha(P^y - P^z)(y)}{\abs{y-z}^{2-\abs{\alpha}}}
			\enskip,\enskip
			\d^\alpha P^{x_S}(x_S)\quad,\quad \abs{\alpha} \leq 1
		\end{equation*}
		for suitable $ x_S, y, z \in S $ in some order, so that
		\begin{equation}
			\eqindent
			\norm{L_w(\vec{P})}_{\ell^1(\R^{d_S})} \approx \Q({\vec{P}} )
			\quad \for \vec{P} \in \wts
			\label{eq.3.2.3}
		\end{equation}
		The construction of such $ L_w $ is based on the technique of \qt{clustering} introduced in \cite{BM07}. See Remark 3.3 of \cite{BM07}. Since $ \#(S) $ is universally bounded, we can compute $ L_w $ from $ S $ using at most $ C $ operations.

		\item Let $ V_f \subset W^2(S) $ be a subspace defined by
		\begin{equation*}
		V_f := \set{ (P^x)_{x \in S} : P^x(x) = 0 \for x \in S\setminus f^{-1}(0) \text{ and } P^x \equiv 0 \for x \in f^{-1}(0) }.
		\end{equation*}
		Let $ \Pi_f = (\Pi_f^x)_{x \in S} : W^2(S) \to V_f $ be the natural projection defined by
		\begin{equation*}
		\Pi_f^x(P^x) = (0,\d_{e_1}P^x, \d_{e_2}P^x).
		\end{equation*}
			
		Let $ \vp_f \in W^2(S) $ denote the vector
		\begin{equation*}
		\vp_f := \brac{f(x),0,0}_{x \in S}.
		\end{equation*}It is clear that $ \Af = \vp_f + V_f $.

		\item Let $ L_f = (L_f^x)_{x \in S} : W^2(S) \to W^2(S) $ be a linear endomorphism defined by $ L_f^x(P^x) = \frac{P^x}{\sqrt{f(x)}} $ for $ x \in S \setminus f^{-1}(0) $ and $ L_f^x \equiv 0 $ for $ x \in f^{-1}(0) $.
 		
	\end{itemize}

	We see that
	\begin{equation}
		\M(\vec{P}) \approx \norm{L_f\Pi_f(\vp)}^2_{\ell^2(\R^{d})}
		\quad \for \vec{P} \in \Af.
		\label{eq.3.2.5}
	\end{equation}

	Combining \eqref{eq.3.2.3} and \eqref{eq.3.2.5}, we see that
	\begin{equation}
		({\Q}+\M)(\vec{P}) \approx  \norm{L_f\Pi_f(\vp)}^2_{\ell^2(\R^{d})}+\norm{L_w(\vp)}_{\ell^1(\R^{d})}
		\quad \for \vec{P} \in \Af = \vp_f + V_f.
		\label{eq.3.2.6}
	\end{equation}
	
	Let $ \beta := L_w(\vp) $ and $ A:= (L_f\Pi_f)^T(L_f\Pi_f) $. We see from \eqref{eq.3.2.1} and \eqref{eq.3.2.6} that computing the order of magnitude of $ \norm{f}_{\ctp(S)} $ amounts to solving the following optimization problem:
	\begin{equation}
		\text{ Minimize }  \beta^t A \beta +\norm{\beta}_{\ell^1(\R^d)}
		\quad \text{ subject to } L_w^{-1}\beta \in \vp_f + V_f.
		\label{eq.qp}
	\end{equation}
	
	We note that \eqref{eq.qp} is a convex quadratic programming problem with affine constraint. We can find the exact solution to \eqref{eq.qp} by solving for the associated Karush-Kuhn-Tucker conditions, which consist of a bounded system of linear equalities and inequalities\cite{BV-CO}. Thus, we can compute the order of magnitude of $ \norm{f}_{\ctp(S)} $ in $ C $ operations. See Appendix \ref{app:QP} for details

	\subsection{Essential convex sets}

	\newcommand{\s}{\sigma}
	\newcommand{\sk}{\sigma^\sharp}
	\newcommand{\G}{\Gamma_+}
	\newcommand{\Gk}{\Gamma_+^\sharp}
	\newcommand{\Gt}{\widetilde{\Gamma}}
	\newcommand{\Gtk}{\widetilde{\Gamma}^\sharp}

	\begin{definition}\label{def.convex}
		Let $ E \subset \rt $ be a finite set. 
		
		\begin{itemize}
			\item For $ x \in \rt $, $ S \subset E $, and $ k \geq 0$, we define
			\begin{equation}
				\eqindent
				\begin{split}
					\s(x,S) &:= \set{\jet_x\phi : \phi \in \ct(\rt),\enskip
						\phi|_S = 0, \text{ and }
						\norm{\phi}_{\ct(\rt)} \leq 1}, \text{ and }\\
					\sk(x,k) &:= \bigcap_{S \subset E, \#(S) \leq k} \s(x,S).
				\end{split}
				\label{eq.sigma-def}
			\end{equation}
			
			\item Let $ f : E \to \pos $ be given. For $ x \in \rt $, $ S \subset E $, $ k \geq 0$, and $ M \geq 0 $, we define
			\begin{equation}
				\eqindent
				\begin{split}
					\G(x,S,M,f) &:= \set{\jet_x F : F \in \ctp(\rt),\enskip
						F|_S = f, \text{ and }
						\norm{F}_{\ct(\rt)} \leq M},\text{ and }\\
					\Gk(x,k,M,f) &:= \bigcap_{S \subset E, \#(S) \leq k}\G(x,S,M,f).
				\end{split}
				\label{eq.G-def}
			\end{equation}

		\end{itemize}

	\end{definition}

	Adapting the proof of the Finiteness Principle for nonnegative $ \ct(\rt) $ interpolation (Theorem 4 of \cite{JL20}), we have the following.
	\begin{lemma}\label{lem.fp-sigma}
		There exists a universal constant $ C $ such that the following holds. Let $ E \subset \rt $ be a finite set. Let $ \sigma $ and $ \sk $ be as in Definition \ref{def.convex}. Then for any $ x \in \rt $, 
		\begin{equation*}
			C^{-1}\cdot \sk(x,16) \subset \sigma(x,E) \subset C\cdot \sigma(x,16).
		\end{equation*}
	\end{lemma}

	\subsection{Callahan-Kosaraju decomposition}

	We will use the data structure introduced by Callahan and Kosaraju\cite{CK95}.

	\begin{lemma}[Callahan-Kosaraju decomposition]\label{lem.wspd}
		Let $ E \subset \Rn $ with $ \#(E) = N < \infty $. Let $ \kappa > 0 $. We can partition $ E \times E \setminus diagonal(E) $ into subsets $ E_1'\times E_1'', \cdots, E_L'\times E_L'' $ satisfying the following.
		\begin{enumerate}[(A)]
			\item $ L \leq C(\kappa,n)N $.
			\item For each $ \ell = 1, \cdots, L $, we have
			\begin{equation*}
				\diam{E_\ell'}, \diam{E_\ell''} \leq \kappa \cdot\dist{E_\ell'}{E_\ell''}\,.
			\end{equation*}
			\item Moreover, we may pick $ x_\ell' \in E_\ell' $ and $ x_\ell'' \in E_\ell'' $ for each $ \ell = 1, \cdots, L $, such that the $ x_\ell', x_\ell'' $ for $ \ell = 1, \cdots, L $ can all be computed using at most $ C(\kappa,n)N\log N $ operations and $ C(\kappa,n)N $ storage.
		\end{enumerate}
		Here, $ C(\kappa,n) $ is a constant that depends only on $ \kappa $ and $ n $.
	\end{lemma}

	\section{Algorithm \ref{alg.norm}: Estimation of trace norm}

	\subsection{Proof of Theorem \ref{thm.sfp}}

	In this section, we prove Theorem \ref{thm.sfp} by assuming Theorem \ref{thm.bd-alg}, whose proof will appear in Section \ref{proof-thm.bd-alg}.
	
	With a slight tweak, the argument in the proof of Lemma 3.1 in \cite{F09-Data-3} yields the following. 
	
	\begin{lemma}\label{lem.wf-rep}
		Let $ E \subset \R^2 $ be a finite set. Let $ \kappa_0 > 0 $ be a constant that is sufficiently small. Let $ E_\ell', E_\ell'' $ be as in Lemma \ref{lem.wspd} with $ \kappa = \kappa_0 $. Suppose $ \vec{P} = (P^x)_{x \in E} \in \wtp(E) $ satisfies the following.
		\begin{enumerate}[(A)]
			\item $ P^x \in \Gamma_+(x,\void,M,f) $ for each $ x \in E $, with $ \Gamma_+ $ as in \eqref{eq.G-def}. 
			\item $ \abs{\d^\alpha (P^{x_\ell'} - P^{x_\ell''})(x_\ell'')} \leq M\abs{x_\ell' - x_\ell''}^{2-\abs{\alpha}}
			\text{ for } \abs{\alpha} \leq 1,\,
			\ell = 1, \cdots, L $.
		\end{enumerate}
		Then $ \norm{\vec{P}}_{\wtp(E)} \leq CM $. 
	\end{lemma}

	Recall Lemma 3.2 of \cite{F09-Data-3}.
	
	\begin{lemma}
		Let $ E \subset \R^2 $ be a finite set. Let $ E_\ell' $ and $ E_\ell'' $ be as in Lemma \ref{lem.wspd} with $ \ell = 1, \cdots, L $. Then every $ x \in E $ arises as an $ x_\ell' $ for some $ \ell \in \set{1, \cdots, L} $.
		\label{lem.arise} 
	\end{lemma}

	We now have all the ingredients for the proof of Theorem \ref{thm.sfp}. 
	
	\begin{proof}[Proof of Theorem \ref{thm.sfp} Assuming Theorem \ref{thm.bd-alg}]
		
		Let $ E \subset \rt $ be a finite set. Let $ \set{\Xi_x, x \in \rt} $ be as in Theorem \ref{thm.bd-alg}. For each $ x \in E $, let $ S(x) $ be as in Remark \ref{rem.Sx}.
		
		Let $ \kappa_0 $ be as in Lemma \ref{lem.wf-rep}. Let $ (x_\ell', x_\ell'') \in E\times E $, $ \ell = 1, \cdots, L $, be as in Lemma \ref{lem.wspd} with $ \kappa = \kappa_0 $. 
		
		We set
		\begin{equation}
			S_\ell := \set{x_\ell', x_\ell''} \cup S(x_\ell') \cup S(x_\ell'')
			\enskip,\enskip
			\ell = 1, \cdots, L.
			\label{eq.Sl-def}
		\end{equation}
		
		Conclusion (A) follows from Theorem \ref{thm.bd-alg}(C), Remark \ref{rem.Sx}, and Lemma \ref{lem.wspd}. 
		
		Conclusion (B) follows from Theorem \ref{thm.bd-alg}(C) and Remark \ref{rem.Sx}.
		
		Conclusion (C) follows from Lemma \ref{lem.wspd}(C). 
		
		Now we verify conclusion (D). We modify the argument in \cite{F09-Data-3}.
		
		Fix $ f : E \to \pos $. Set
		\begin{equation}
			M := \max_{\ell = 1, \cdots, L}\norm{f}_{\ctp(S_\ell)}.
			\label{eq.M-fl}
		\end{equation}

		Thanks to \eqref{eq.M-fl}, we see that $ \norm{f}_{\ctp(S_\ell)} \leq M $ for $ \ell = 1, \cdots, L $. Thus, for each $ \ell = 1, \cdots, L $, there exists $ F_\ell \in \ctp(\rt) $ such that \begin{equation}
			\norm{F_\ell}_{\ct(\rt)} \leq 2M \text{ and } F_\ell(x) = f(x)\for x \in S_\ell.
			\label{eq.3.6}
		\end{equation}
		Fix such $ F_\ell $. For $ \ell = 1, \cdots, L $, we define 
		\begin{equation}
			f_\ell :E \to \pos 
			\text{ by }
			f_\ell(x) := F_\ell(x) \for x \in E.
			\label{eq.3.7}
		\end{equation}
		From \eqref{eq.3.6} and \eqref{eq.3.7}, we see that
		\begin{equation}
			\norm{f_\ell}_{\ctp(E)} \leq 2M \for \ell = 1, \cdots, L.
			\label{eq.f-l-est}
		\end{equation}

		For each $ \ell = 1, \cdots, L $, we define \begin{equation}
			P_\ell' := \jet_{x_\ell'}\brac{\Xi_{x_\ell'}(f_\ell,2M)}
			\text{ and }
			P_\ell'' := \jet_{x_\ell''}\brac{\Xi_{x_\ell''}(f_\ell,2M)}.
			\label{eq.P-ell-def}
		\end{equation}
		
		We will show that the assignment \eqref{eq.P-ell-def} unambiguously defines a Whitney field over $ E $.

		\begin{claim}\label{claim.ell}
			Let $ \ell_1, \ell_2 \in \set{1, \cdots, L} $.
			\begin{enumerate}[(a)]
				\item Suppose $ x_{\ell_1}' = x_{\ell_2}' $. Then $ P_{\ell_1}' = P_{\ell_2}' $.
				\item Suppose $ x_{\ell_1}'' = x_{\ell_2}'' $. Then $ P_{\ell_1}'' = P_{\ell_2}'' $.
				\item Suppose $ x_{\ell_1}' = x_{\ell_2}'' $. Then $ P_{\ell_1}' = P_{\ell_2}'' $.
			\end{enumerate}
			
		\end{claim}
		
		\begin{proof}[Proof of Claim \ref{claim.ell}]
			We prove (a). The proofs for (b) and (c) are similar.
			
			Suppose $ x_{\ell_1}' = x_{\ell_2}' =: x_0 $. Let $ S(x_0) $ be as in Remark \ref{rem.Sx}. By \eqref{eq.Sl-def}, we see that
			\begin{equation*}
				S(x_0) \subset S_{\ell_1} \cap S_{\ell_2}. 
			\end{equation*}
			Therefore, we have
			\begin{equation*}
				f_{\ell_1}(x) = f_{\ell_2}(x) \for x \in S(x_0).
			\end{equation*}
			Thanks to Theorem \ref{thm.bd-alg}(A), Remark \ref{rem.Sx}, and \eqref{eq.f-l-est}, we see that
			\begin{equation*}
				\Xi_{x_0}(f_{\ell_1},2M) = 	\Xi_{x_0}(f_{\ell_2},2M).
			\end{equation*}
			By \eqref{eq.P-ell-def}, we see that $ P_{\ell_1} = P_{\ell_2} $. This proves (a).
		\end{proof}

		By Lemma \ref{lem.arise}, there exists a pair of maps:
		\begin{equation}
			\begin{split}
				\text{A surjection } \pi &: \set{1, \cdots, L} \to E
				\text{ such that }
				\pi(\ell) = x_{\ell}'
				\for \ell = 1, \cdots, L\text{, and }\\
				{\text{An injection }} \rho &: E \to \set{1, \cdots, L}
				\text{ such that }x_{\rho(x)}' = x \for x \in E, \text{ i.e., } \pi \circ \rho = id_E. 
			\end{split}
			\label{eq.lx}
		\end{equation}
		The surjection $ \pi $ is determined by the Callahan-Kosaraju decomposition (Lemma \ref{lem.wspd}), but the choice of $ \rho $ is not necessarily unique. 
		
		Thanks to Claim \ref{claim.ell} and the fact that $ E_\ell' \times E_\ell'' \subset E \times E \setminus diagonal(E) $, assignment \eqref{eq.P-ell-def} produces for each $ x \in E $ a uniquely defined polynomial
		\begin{equation}
			P^x = \jet_{x}\brac{\Xi_{x}(f_{\rho(x)},2M)},
			\label{eq.3.9}
		\end{equation}
		with $ \Xi_x $ as in Theorem \ref{thm.bd-alg} and $ \rho(x) $ as in \eqref{eq.lx}. Note that, as shown in Claim \ref{claim.ell}, the polynomial $ P^x $ in \eqref{eq.3.9} is independent of the choice of $ \rho $ as a right-inverse of $ \pi $ in \eqref{eq.lx}.
		
		Thanks to Theorem \ref{thm.bd-alg}(B) and \eqref{eq.f-l-est}--\eqref{eq.3.9}, for each $ \ell = 1, \cdots, L $, there exists a function $ \tilde{F}_\ell \in \ct(\rt) $ such that
		\begin{equation}\label{eq.F-ell-1}
			\norm{\tilde{F}_\ell}_{\ct(\rt)} \leq CM
			\text{ and }
			\tilde{F}_\ell \geq 0 \text{ on }\rt;
		\end{equation}
		\begin{equation}\label{eq.F-ell-2}
			\tilde{F}_\ell = f_{\ell}(x) = f(x)
			\for x \in S_{\ell} ; \text{ and }
		\end{equation}
		\begin{equation}\label{eq.F-ell-3}
			\jet_{x_{\ell}'}\tilde{F}_\ell = P^{x_{\ell}'} = \jet_{x_{\ell}'}\brac{ \Xi_{x_{\ell}'} (f_{\ell},2M) }, \text{ and }
			\jet_{x_{\ell}''}\tilde{F}_\ell = P^{x_{\ell}''} = \jet_{x_{\ell}''}\brac{ \Xi_{x_{\ell}''} (f_{\ell},2M) }.
		\end{equation}

		Thanks to \eqref{eq.F-ell-1} and \eqref{eq.F-ell-2}, we have 
		\begin{equation}
			P^{x_\ell'} \in \Gamma_+(x_\ell',\{x_\ell'\},CM,f)
			\for \ell = 1, \cdots, L.
			\label{eq.3.15}
		\end{equation}
		Thanks to \eqref{eq.F-ell-1} and \eqref{eq.F-ell-3}, we have 
		\begin{equation}
			\abs{\d^\alpha (P^{x_\ell'} - P^{x_\ell''})(x_\ell'')}
			\leq CM\abs{x_\ell' - x_\ell''}^{2-\abs{\alpha}}
			\for \abs{\alpha} \leq 1, \ell = 1, \cdots, L.
			\label{eq.3.16}
		\end{equation}

		Therefore, by Lemma \ref{lem.wf-rep}, \eqref{eq.3.15}, and \eqref{eq.3.16}, the Whitney field $ \vec{P} = (P^x)_{x \in E} $, with $ P^x $ as in \eqref{eq.3.9}, satisfies 
		\begin{equation*}
			\vec{P} \in \wtp(E),
			P^x(x) = f(x) \for x \in E,
			\text{ and }
			\norm{\vec{P}}_{\wtp(E)} \leq CM.
		\end{equation*}
		By Lemma \ref{lem.WT}(B), there exists a function $ F \in \ctp(\rt) $ such that $ \norm{F}_{\ct(\rt)} \leq CM $ and $ \jet_x F = P^x $ for each $ x \in E $. In particular, $ F(x) = P^x(x) = f(x) $ for each $ x \in E $. Thus, $ \norm{f}_{\ctp(E)} \leq CM $. This proves conclusion (D).
		
		Theorem \ref{thm.sfp} is proved. 
		
	\end{proof}

	\subsection{Explanation of Algorithm \ref{alg.norm}}

	Below are the steps of Algorithm \ref{alg.norm}. 
	
	\begin{enumerate}[Step 1.]
		\item Compute $ S_1, \cdots, S_L $ from $ E $ as in Theorem \ref{thm.sfp}.
		
		\item Read $ f : E \to \pos $.
		
		\item For $ \ell = 1, \cdots, L $, compute a number $ M_\ell $ such that $ M_\ell $ has the same order of magnitude as $ \norm{f}_{\ctp(S_\ell)} $.

		\item Return $ M := \max \set{M_\ell : \ell = 1, \cdots, L} $. 
	\end{enumerate}

	The number $ M $ produced in Step 4 has the same order of magnitude as $ \norm{f}_{\ctp(E)} $, thanks to Theorem \ref{thm.sfp} and Lemma \ref{lem.WT}. Therefore, Algorithm \ref{alg.norm} accomplishes what we claim to do.
	
	We now analyze the complexity of Algorithm \ref{alg.norm}.

	By Theorem \ref{thm.sfp}, Step 1 requires no more than $ CN\log N $ operations and $ CN $ storage. 
	
	Step 3 requires no more than $ CN $ operations. Indeed, on one hand, computing each $ M_\ell $ requires no more than $ C $ operations, thanks to the discussion in Section \ref{sect:lasso}; on the other hand, we need to carry out $ L $ computations, with $ L \leq CN $.
	
	Finally, Step 4 requires no more than $ CN $ operations.

	This concludes our discussion of Algorithm \ref{alg.norm}.

	\section{Approximation of $ \sigma^\sharp $}
	
	\newcommand{\rbn}{\mathbb{R}^{\bar{N}}}
	\newcommand{\cte}{\ct(E)}
	
	This and the next sections will be devoted to the proof of Theorem \ref{thm.bd-alg}. To prepare the way, in this section, we introduce the relevant objects and show how they can be computed efficiently. 
	
	We begin by reviewing some key objects introduced in \cite{FK09-Data-1,FK09-Data-2}, which we will use to effectively approximate the shapes of $ \sk(x,16) $ for $ x \in E $. 
	
	We will be working with $ \ct(\rt) $ functions instead of $ \ctp(\rt) $ functions.

	\subsection{Parameterized approximate linear algebra problems (PALP)}

	Let $ \bar{N} \geq 1 $. Let $ \set{\xi_1, \cdots, \xi_{\bar{N}}} $ be the standard basis for $ \R^{\bar{N}} $. We recall the following definition in Section 6 of \cite{FK09-Data-2}.
	
	\begin{definition}\label{def.palp}
		A \underline{parameterized approximate linear algebra problem (PALP for short)} is an object of the form:
		\begin{equation}
			\A = \left[ (\ul{\lambda}_1, \cdots, \ul{\lambda}_{i_{\max}}), (\ul{b}_1, \cdots, \ul{b}_{i_{\max}}), (\epsilon_1, \cdots, \epsilon_{i_{\max}}) \right],
			\label{eq.PALP-def}
		\end{equation}
		where \begin{itemize}
			\item Each $ \ul{\lambda}_i $ is a linear functional on $ \P $, which we will refer to as a ``linear functional'';
			\item Each $ \ul{b}_i $ is a linear functional on $ \cte $, which we will refer to as a ``target functional''; and
			\item Each $ \epsilon_i \in \pos $, which we will refer to as a ``tolerance''.
		\end{itemize}
		Given a PALP $ \A $ in the form \eqref{eq.PALP-def}, we introduce the following terminologies:
		\begin{itemize}
			\item We call $ i_{\max} $ the \ul{length} of $ \A $;
			\item We say $ \A $ \ul{has depth $ D $} if each of the linear functionals $ \ul{b}_i $ on $ \R^{\bar{N}} $ has depth $ D $ with respect to the basis $ \set{\xi_1, \cdots, \xi_{\bar{N}}} $ (see Definition \ref{def.depth}).
		\end{itemize}

	\end{definition}

	Recall Definition \ref{def.depth}. We assume that every PALP is ``efficiently stored'', namely, each of the target functionals are stored in its efficient representation. In particular, given a PALP $ \A $ of the form \eqref{eq.PALP-def} and a target $ \ul{b}_i $ of $ \A $, we have access to a set of indices $ \set{i_1, \cdots, i_D} \subset \set{1,\cdots,N} $, such that $ \ul{b}_i $ is completely determined by its action on $ \set{\xi_{i_1}, \cdots, \xi_{i_D}} \subset \set{\xi_1, \cdots, \xi_N} $. Here $ i_D = \depth(\ul{b}_i) $. We define
	\begin{equation}
		S(\ul{b}_i):= \set{x_{i_1}, \cdots, x_{i_D}} \subset E.
		\label{eq.S(b)-def}
	\end{equation}
	Given a PALP of the form \eqref{eq.PALP-def}, we define
	\begin{equation}
		S(\A) := \bigcup_{i = 1}^{i_{\max}}S(\ul{b}_i) \subset E
		\label{eq.S(A)-def}
	\end{equation}
	with $ S(\ul{b}_i) $ as in \eqref{eq.S(b)-def}.

	\subsection{Blobs and PALPs}
	\newcommand{\K}{\mathcal{K}}

	\begin{definition}
		A \ul{blob} in $ \P $ is a family $ \vec{\K} = (\K_M)_{M \geq 0} $ of (possibly empty) convex subsets $ \K_M \subset V $ parameterized by $ M \in \pos $, such that $ M < M' $ implies $ \K_M \subseteq \K_{M'} $. We say two blobs $ \vec{\K} = (\K_M)_{M \geq 0} $ and $ \vec{\K}' = (\K'_M)_{M \geq 0} $ are \ul{$ C $-equivalent} if $ \K_{C^{-1}M}\subset \K'_M \subset \K_{CM} $ for each $ M \in \pos $.
	\end{definition}

	Let $ \A $ be a PALP of the form \eqref{eq.PALP-def}. For each $ \phi \in \cte \cong \R^{\#(E)} $, we have a blob defined by
	\begin{equation}
		\begin{split}
			\vec{\K}_\phi(\A)  &= \brac{\K_\phi(\A,M)}_{M \geq 0}
			\text{, where}\\
			\K_\phi(\A,M) &:= \set{P \in \P : \abs{\ul{\lambda}_i(P) - \ul{b}_i(\phi)} \leq M\epsilon_i\for i = 1, \cdots, i_{\max} } \subset V.
		\end{split}
		\label{eq.K1}
	\end{equation}
	In this paper, we will be mostly interested in the centrally symmetric (called \qt{homogeneous} in \cite{FK09-Data-2}) polytope defined by setting $ \phi \equiv 0 $:
	\begin{equation}
		\sigma(\A) := \K_0(\A,1).
		\label{eq.sigma(A)-def}
	\end{equation}
	Note that $ \sigma(\A) $ is never empty, since it contains the zero polynomial.

	\subsection{Essential PALPs and Blobs}

	Let $ E \subset \rt $ be a finite set with $ \#(E) = N $. We assume that $ E $ is labeled: $ E = \set{x_1, \cdots, x_N}  $. We identify $ \ct(E) \cong \R^N $ with respect to the standard basis $ \set{\xi_1, \cdots, \xi_N} $ for $ \R^N $. 
	
	\begin{definition}
		\renewcommand{\S}{\Sigma} For each $ x \in \rt $ and $ \phi \in \ct(E) $, we define a blob
		\begin{equation}
			\begin{split}
				\vec{\S}_\phi(x) &= \brac{\S_\phi(x,M)}_{M\geq 0}\text{ where}\\
				\S_\phi(x,M) &:= \set{ P \in \P : \begin{matrix}
						\text{ There exists } G \in \ct(\rt) \text{ with}\\\norm{G}_{\ct(\rt)} \leq M, G|_E = \phi, \text{ and }\jet_x G= P.
				\end{matrix} }
			\end{split}
			\label{eq.Sigma(x)}
		\end{equation}
	\end{definition}

	It is clear from the definition of $ \sigma $ in \eqref{eq.sigma-def} that
	\begin{equation*}
		\sigma(x,E) = \Sigma_0(x,1).
	\end{equation*}
	Therefore, thanks to Lemma \ref{lem.fp-sigma}, we have
	\begin{equation}
		C^{-1}\cdot\sk(x,16) \subset \Sigma_0(x,1) \subset C\cdot\sk(x,16), \enskip x \in E
		\label{eq.sk-Sigma_0}
	\end{equation}
	for some universal constant $ C $.

	We summarize some relevant results from \cite{FK09-Data-2}.
	
	\begin{lemma}\label{lem.FK-palp}
		Let $ E \subset \rt $ be finite. Using at most $ CN\log N $ operations and $ CN $ storage, we can compute a list of PALPs $ \set{\A(x) : x \in E} $ such that the following hold.
		\begin{enumerate}[(A)]
			\item There exists a universal constant $ D_0 $ such that for each $ x \in E $, $ \A(x) $ has length no greater than $ 3 = \dim \P $ and has depth $ D_0 $.
			\item For each given $ x \in \rt $ and $ \phi \in \ct(E) $, the blobs $ \vec{\K}_\phi(\A(x)) $ as in \eqref{eq.K1} and $ \vec{\Sigma}_\phi(x) $ as in \eqref{eq.Sigma(x)} are $ C $-equivalent. 
		\end{enumerate}
	\end{lemma}

	See Section 11 of \cite{FK09-Data-2} for Lemma \ref{lem.FK-palp}(A), and Sections 10, 11, and Lemma 34.3 of \cite{FK09-Data-2} for Lemma \ref{lem.FK-palp}(B). 
	
	The main lemma of this section is the following.
	
	\begin{lemma}\label{lem.sigma-main}
		Let $ E \subset \rt $ be given. Let $ \set{\A(x) : x \in E} $ be as in Lemma \ref{lem.FK-palp}. Recall the definitions of $ \sigma $ and $ S(\A(x)) $ as in \eqref{eq.sigma-def} and \eqref{eq.S(A)-def}. Then there exists a universal constant $ C $ such that, for each $ x \in E $, 
		\begin{equation*}
			C^{-1}\cdot \sigma(x,S(\A(x))) \subset \sk(x,16) \subset C\cdot \sigma(x,S(\A(x))).
		\end{equation*}
	\end{lemma}
	
	\begin{proof}
		For centrally symmetric $ \sigma,\sigma' \subset \P $, we write $ \sigma \approx \sigma' $ if there exists a universal constant $ C $ such that $ C^{-1}\cdot \sigma \subset \sigma' \subset C\cdot \sigma $. Thus, we need to show $ \sigma(x,\A(x)) \approx \sk(x,16) $ for $ x \in E $. 
		
		Thanks to Lemma \ref{lem.fp-sigma}, Lemma \ref{lem.FK-palp}(B) (applied to $ \phi \equiv 0 $), \eqref{eq.sigma(A)-def}, and \eqref{eq.sk-Sigma_0}, we have
		\begin{equation}
			\sk(x,16) \approx \sk(x,E) \approx \K_0(\A(x),1) = \sigma(\A(x))\for x \in E.
			\label{eq.3.8}
		\end{equation}
		Therefore, it suffices to show that
		\begin{equation*}
			\sigma(x,S(\A(x))) \approx \sigma(\A(x))
			\for x \in E.
		\end{equation*}
		From \eqref{eq.3.8} and the definition of $ \sigma $ in \eqref{eq.sigma-def}, we see that
		\begin{equation*}
			\sigma(\A(x)) \subset C\cdot \sigma(x,E) \subset C\cdot \sigma(x,S(\A(x))).
		\end{equation*}
		It remains to show that
		\begin{equation*}
			\sigma(x,S(\A(x))) \subset C\cdot \sigma(\A(x)) .
		\end{equation*}
		
		Let $ x \in E $ and let $ P \in \sigma(x,S(\A(x))) $. Then there exists $ \phi \in \ct(\rt) $ such that $ \norm{\phi}_{\ct(\rt)} \leq 1 $, $ \phi(x) = 0 $ for all $ x \in S(\A(x)) $, and $ \jet_x(\phi) = P $. Note that $ \phi|_E \in \ct(E) $. We abuse notation and write $ \phi $ in place of $ \phi|_E $ when there is no possibility of confusion.
		
		It is clear from the definition of $ \Sigma_\phi(x,M) $ in \eqref{eq.Sigma(x)} that
		\begin{equation*}
			P \in \Sigma_\phi(x,1).
		\end{equation*}
		By Lemma \ref{lem.FK-palp}(B), we have
		\begin{equation*}
			P \in \K_\phi(\A(x),C)
		\end{equation*}
		with $ \K_\phi(\A(x),C) $ as in \eqref{eq.K1}. In particular, we have
		\begin{equation}
			\abs{\ul{\lambda}_i(P) - \ul{b}_i(\phi)} \leq C\epsilon_i
			\for i = 1, \cdots, L = \mathrm{length}(\A(x)).
			\label{eq.3.14}
		\end{equation}
		Here, the $ \ul{\lambda}_1, \cdots, \ul{\lambda}_L$,$ \ul{b}_1, \cdots, \ul{b}_L $, and $ \epsilon_1, \cdots, \epsilon_L $, respectively, are the linear functionals, target functionals, and the thresholds of $ \A(x) $. However, by the definition of $ S(\A(x)) $ in \eqref{eq.S(A)-def} and the fact that $ \phi \equiv 0 $ on $ S(\A(x)) $, we see that \eqref{eq.3.14} simplifies to 
		\begin{equation*}
			\abs{\ul{\lambda}_i(P)} \leq C\epsilon_i
			\for i = 1, \cdots, L = \mathrm{length}(\A(x)).
		\end{equation*}
		This is equivalent to the statement
		\begin{equation*}
			P \in \K_0(\A(x),C) = C\cdot \sigma(\A(x)).
		\end{equation*}
		Lemma \ref{lem.sigma-main} is proved. 
	\end{proof}

	\section{Algorithm \ref{alg.interpolant}: Computing a $C$-optimal interpolant}

	Let $ E \subset \rt $ be a finite set. We fix $ E $ throughout the rest of the paper.

	\subsection{Calder\'on-Zygmund squares}

	Let $ \tilde{\sigma} \subset \rt $ be a symmetric convex set. We define
	\begin{equation}\label{eq.diam-def}
		\diam \tilde{\sigma} := 2\cdot \sup_{u \in \rt, \abs{u} = 1}p_{\tilde{\sigma}}(u),
	\end{equation}
	where $ p_{\tilde{\sigma}}(u) $ is a gauge function given by
	\begin{equation}
		p_{\tilde{\sigma}}(u) := \sup \set{r \geq 0: ru\subset{\tilde{\sigma}}}. 
		\label{eq.Mink-def}
	\end{equation}

	Let $ \set{\A(x):x \in E} $ be as in Lemma \ref{lem.FK-palp}, and let $ \sigma(\A(x)) \subset \P $ be as in \eqref{eq.sigma(A)-def}. Note that for each $ x \in E $, $ \sigma(\A(x)) \subset \P $ two-dimensional. Indeed, thanks to Lemma \ref{lem.FK-palp}(B) (with $ \phi \equiv 0 $), any $ P \in \sigma(\A(x)) $, $ x \in E $, must have $ P(x) = 0 $. Thus, for each $ x \in E $, we can identify $ \sigma(\A(x)) $ as a subset of $ \rt $ via the map
	\begin{equation}
		\sigma(\A(x))\ni P \mapsto (\grad P \cdot e_1, \grad P \cdot e_2),
		\label{eq.jet-id}
	\end{equation}
	where $ \set{e_1, e_2} $ is the chosen orthonormal system.

	\newcommand{\Lz}{\Lambda_0}
	\newcommand{\xqs}{{x_Q^\sharp}}
	
	Let $ A_1, A_2 > 0 $ be sufficiently large dyadic numbers. Let $ \set{\A(x):x \in E} $ be as in Lemma \ref{lem.FK-palp}. We say a dyadic square \ul{$ Q $ is OK} if the following hold.
	\begin{itemize}
		
		\item Either $ \#(E \cap 5Q) \leq 1 $, or $ \diam\sigma(\A(x)) \geq A_1\dq $ for all $ x \in E \cap 5Q $. Here and below, the $ \diam(\sigma(\A(x))) $ is defined using the formula \eqref{eq.diam-def} via the identification \eqref{eq.jet-id}. 
		
		\item $ \dq \leq A_2^{-1} $. 
	\end{itemize}

	\begin{definition}\label{def.CZ}
		We write $ \Lz $ to denote the collection of dyadic squares $ Q $ such that both of the following hold.
		\begin{enumerate}[(A)]
			\item $ Q $ is OK (see above).
			\item Suppose $ \dq < A_2^{-1} $, then $ Q^+ $ is not OK.
		\end{enumerate}
	\end{definition}

	\begin{remark}
		Note that there are two differences in the definition of $ \Lz $ than those in \cite{JL20,JL20-Ext}.
		\begin{itemize}
			\item We use $ 5Q $ in the definition of $ \Lz $ instead of using $ 2Q $. This has the advantage that $ 5Q^+ \subset 5^2Q = 25Q $. 
			\item We do not require $ \diam\sigma(\A(x)) \geq A_1\dq $ for $ x \in E \cap 5Q $ when $ \#(E \cap 5Q) = 1 $.
		\end{itemize}
		We will provide explanation when these differences change the structure of the analysis. Otherwise, we will simply add the word ``variant'' to our reference to results in \cite{JL20,JL20-Ext}.
	\end{remark}

	\begin{lemma}\label{lem.CZ}
		$ \Lz $ enjoys the following properties. 
		\begin{enumerate}[(A)]
			\item\label{lem.CZ-cover} $ \Lz $ forms a cover of $ \rt $ with good geometry:
			\begin{enumerate}[(\text{A}1)]
				\item $ \rt = \bigcup_{Q \in \Lz} Q $;
				\item If $ Q, Q' \in \Lz $ with $ (1+2c_G)Q \cap (1+2c_G)Q' \neq \void $, then 
				\begin{equation*}\eqindent
					C^{-1}\dq \leq \delta_{Q'} \leq C\dq;
				\end{equation*}
				and as a consequence, for each $ Q \in \Lz $, 
				\begin{equation*}\eqindent
					\#\set{Q'\in \Lz: (1+c_G)Q' \cap (1+c_G)Q \neq \void} \leq C'.
				\end{equation*}
				Here, $ C,C' $ are universal constants, and $ c_G $ is a sufficiently small constant, say $ 1/32 $.
			\end{enumerate}

			\item\label{lem.CZ-graph}Let $ Q \in \Lz $. Then there exists $ \phi \in \ct(\R) $ such that 
			\begin{equation}
				\eqindent
				\rho(E \cap 5Q) \subset \set{(t,\phi(t)):t \in \R},
				\label{eq.phi-1}
			\end{equation}
			where $ \rho $ is some rotation about the origin depending only on $ Q $.
			Moreover, $ \phi $ satisfies the estimates
			\begin{equation}
				\eqindent
				\abs{\ddtm\phi(t)} \leq CA_1^{-1}\dq^{1-m}
				\for m = 1,2,
				\label{eq.phi-2}
			\end{equation}
			with $ A_1 $ as in Definition \ref{def.CZ}.
			Furthermore, suppose for some $ x_0 \in E \cap 5Q $ and a unit vector $ u_0 $, we have
			\begin{equation*}
				\diam\sigma(\A(x_0))=p_{\sigma(\A(x_0))}(u_0) 
			\end{equation*}
			with $ \diam \sigma(\A(x_0)) $ and $ p_{\sigma(\A(x_0))}(u_0) $ as in \eqref{eq.diam-def} and \eqref{eq.Mink-def}. Then we can take $ \phi $ to satisfy the following property:
			\begin{enumerate}[(B1)]
				\item We can take $ \rho $ in \eqref{eq.phi-1} to be the rotation specified by $ u_0 \mapsto e_2 $;
				\item We can take $ x_0 = (0,\phi(0)) $.
			\end{enumerate}
			
			As a consequence, there exists a $ C^2 $-diffeomorphism $ \Phi : \rt \to \rt $ defined by
			\begin{equation*}
				\Phi\circ\rho(t_1, t_2) = (t_1, t_2 -\phi(t_1))
				\text{ where $ \rho $ is the rotation as in \eqref{eq.phi-1},}
			\end{equation*}
			such that $ \Phi(E \cap 5Q) \subset \R\times\set{t_2 = 0} $ and $ \abs{\grad^m\Phi}, \abs{\grad^m\Phi^{-1}} \leq CA_1^{-1}\dq^{1-m}$ for $ m = 1,2 $, with $ A_1 $ as in Definition \ref{def.CZ}.
			
		\end{enumerate}
	\end{lemma}

	\begin{remark}
		Lemma \ref{lem.CZ}(A) can be found in Section 21 of \cite{FK09-Data-2}. See also Lemma 5.1 of \cite{JL20}. Lemma \ref{lem.CZ}(B) follows from the proofs of Lemma 5.4 and 5.5 with a minor modification: For $ Q \in \Lz $ with $ \#(E \cap 5Q) \leq 1 $, we can simply take $ \phi $ to be a suitable constant function on $ \R $.  
	\end{remark}

	We recall the following results from \cite{FK09-Data-2}.
	
	\newcommand{\erep}{Rep}

	\begin{lemma}\label{lem.FK-CZ}
		After one-time work using at most $ CN\log N $ operations and $ CN $ storage, we can perform each of the the following tasks using at most $ C\log N $ operations.
		\begin{enumerate}[(A)]
			\item {\rm (Section 26 of \cite{FK09-Data-2})} Given a point $ x \in \rt $, we compute a list $ \Lambda(x):=\set{Q \in \Lz : (1+c_G)Q \ni x} $.
			
			\item {\rm (Section 27 of \cite{FK09-Data-2})} Given a dyadic square $ Q\subset \rt $, we can compute $ Empty(Q) $, with $ Empty(Q) = True $ if $ E \cap 25Q = \void $, and $ Empty(Q) = False $ if $ E \cap 25Q \neq \void $. 
			
			\item {\rm (Section 27 of \cite{FK09-Data-2})} Given a dyadic square $ Q \subset \rt $ with $ E \cap 25Q \neq \void $, we can compute $ \erep(Q) \in E \cap 25Q $, with the property that $ \erep(Q) \in E \cap 5Q $ if $ E \cap 5Q \neq \void $. 
		\end{enumerate}

	\end{lemma}

	\newcommand{\Le}{{\Lambda_{\rm empty}}}
	\newcommand{\Lsk}{\Lambda^{\sharp\sharp}}

	\begin{definition}
		We define the following subcollections of $ \Lz $:
		\LAQ{eq.Lsk-def}{$ \Lsk:= \set{Q \in \Ls : E \cap (1+c_G)Q \neq \void } $}
		
		\LAQ{eq.Ls-def}{$ \Ls:= \set{Q \in \Lz : E \cap 5Q \neq \void} $;}
		
		\LAQ{eq.Le-def}{$ \Le:=\set{Q \in \Lz \setminus \Ls : \dq < A_2^{-1}} $ with $ A_2 $ as in Definition \ref{def.CZ}.}
		
	\end{definition}

	We can think of $ \Lsk $ as the collection of squares with the most \qt{concentrated} information, $ \Ls $ as the largest collection of squares that contain information while still having good local geometry, and $ \Le $ as the collection of squares that do not contain information in their five-time dilation, but are sufficiently small to detect nearby accumulation of points in $ E $.

	We begin with the analysis of $ \Le $ and $ \Ls $.

	\begin{lemma}\label{lem.Ls-Le}
		After one-time work using at most $ CN\log N $ operations and $ CN $ storage, we can perform the following task using at most $ C\log N $ operations: Given $ Q\in \Lz $, we can decide if $ Q \in \Ls $, $ Q \in \Le $, or $ Q \in \Lz\setminus (\Ls \cup\Le) $. 
	\end{lemma}
	
	\begin{proof}
		This is a direct application of Lemma \ref{lem.FK-CZ}(B,C) to $ Q $. 
	\end{proof}

	The next lemma tells us how to relay information to squares in $ \Le $.

	\begin{lemma}\label{lem.mu}
		We can compute a map
		\begin{equation}
			\mu:\Le \to \Ls
			\label{eq.mu-1}
		\end{equation}
		that satisfies
		\begin{equation}
			(1+c_G)\mu(Q) \cap 25Q \neq \void
			\for Q \in \Le\,.
			\label{eq.mu-2}
		\end{equation}
		The one-time work uses at most $ CN\log N $ operations and $ CN $ storage. After that, we can answer queries using at most $ C\log N $ operations. A query consists of a square $ Q \in \Le $, and the response to the query is another square $ \mu(Q) $ that satisfies \eqref{eq.mu-2}. 
		
	\end{lemma}

	\begin{proof}
		
		Suppose $ Q \in \Le $. Then we have $ E \cap 5Q^+ \neq \void $. By the geometry of $ \Lz $, we have $ 5Q^+ \subset 25Q $. Hence, $ E \cap 25Q \neq \void $. Therefore, the map $ \erep $ in Lemma \ref{lem.FK-CZ}(C) is defined for $ Q $. 
		
		We set
		\begin{equation}
			x :=  \erep(Q) \subset E \cap 25Q,
			\label{eq.5.3.2}
		\end{equation}
		with $ \erep $ as in Lemma \ref{lem.FK-CZ}. Note that $ x \notin 5Q $, since $ Q \in \Le $. 
		
		Let $ \Lambda(x) \subset \Lz $ be as in Lemma \ref{lem.FK-CZ}(A). Let $ Q' \in \Lambda(x) $. By the defining property of $ \Lambda(x) $ and the fact that $ x \in E $, we have $ Q' \in \Ls $. Set
		\begin{equation*}
			\mu(Q) := Q' \in \Ls.
		\end{equation*}
		By the previous comment, we have
		\begin{equation}
			(1+c_G)\mu(Q) \ni x.
			\label{eq.5.3.1}
		\end{equation}
		Combining \eqref{eq.5.3.2} and \eqref{eq.5.3.1}, we see that $ (1+c_G)\mu(Q)\cap 25Q \neq \void $. \eqref{eq.mu-2} is satisfied.
		
		By Lemma \ref{lem.FK-CZ}(A,C), the tasks $ \Lambda(\cdot) $ and $ \erep(\cdot) $ require at most $ C\log N $ operations, after one-time work using at most $ CN\log N $ operations and $ CN $ storage. Therefore, computing $ \mu(Q) $ requires at most $ C\log N $ operations, after one-time work using at most $ CN\log N $ operations and $ CN $ storage.

		This proves Lemma \ref{lem.mu}.
	\end{proof}

	\begin{lemma}\label{lem.uQ}
		After one-time work using at most $ CN\log N $ operations and $ CN $ storage, we can perform the following task using at most $ C\log N $ operations: Given $ Q \in \Ls $, compute a pair of unit vectors $ u_Q, u_Q^\perp \in \rt $, such that the following hold.
		\begin{enumerate}[(A)]
			\item $ u_Q $ is orthogonal to $ u_Q^\perp $, and the orthogonal system $ [u_Q^\perp, u_Q] $ has the same orientation as $ [e_1, e_2] $.
			\item Let $ \rho $ be the rotation about the origin specified by $ u_Q \mapsto e_2 $, then there exists a function $ \phi \in \ct(\R) $ that satisfies \eqref{eq.phi-1} and \eqref{eq.phi-2} with this particular $ \rho $. 
		\end{enumerate}
		
	\end{lemma}

	\begin{proof}
		Fix $ Q \in \Ls $. This means that $ E \cap 5Q \neq \void $. In particular, $ \erep(Q) $ is defined, and by Lemma \ref{lem.FK-CZ}(C),
		\begin{equation*}
			x_0 := \erep(Q) \in E \cap 5Q.
		\end{equation*}
		Computing $ x_0 $ requires at most $ C\log N $ operations, after one-time work using at most $ CN\log N $ operations and $ CN $ storage. 
		
		Let $ \A(x_0) $ be as in Lemma \ref{lem.FK-palp}, and let $ \sigma(\A(x_0)) $ be as in \eqref{eq.sigma(A)-def}. By Lemma \ref{lem.FK-palp}(B) (with $ \phi \equiv 0 $), any $ P \in \sigma(\A(x_0)) $ must satisfy $ P(x_0) = 0 $. by Lemma \ref{lem.FK-palp}(A) and definitions \eqref{eq.K1}, \eqref{eq.sigma(A)-def} of $ \sigma(\A(x_0)) $, we see that $ \sigma(\A(x_0)) $ is a two-dimensional parallelogram in $ \P $ centered at the zero polynomial. Therefore, we have
		\begin{equation*}
			\diam \sigma(\A(x_0)) = length(\Delta_0),
		\end{equation*}
		where $ \diam $ is defied in \eqref{eq.diam-def} and $ \Delta_0 $ is the longer diagonal of $ \sigma(\A(x_0)) $.
		
		Set $ u_Q $ to be a unit vector parallel to $ \Delta_0 $. Lemma \ref{lem.uQ}(B) then follows from Lemma \ref{lem.CZ}(B).
		
		We compute another vector $ u_Q^\perp $ such that $ \set{u_Q, u_Q^\perp} $ satisfies Lemma \ref{lem.uQ}(A). Computing $ \set{u_Q, u_Q^\perp} $ from $ \sigma(\A(x_0)) $ uses elementary linear algebra, and requires at most $ C $ operations.
		
		Lemma \ref{lem.uQ} is proved.

	\end{proof}

	\begin{lemma}\label{lem.rep}
		After one-time work using at most $ CN\log N $ operations and $ CN $ storage, we can perform the following task using at most $ C\log N $ operations: Given $ Q \in \Lz $, we can compute a point $ \xqs \in Q $ such that
		\begin{equation}
			\dist{\xqs}{E} \geq c_0\dq
			\label{eq.xqs}
		\end{equation}
		for some universal constant $ c_0 \geq 0 $.
	\end{lemma}
	
	\begin{proof}
		Let $ Q \in \Lz$ be given. 
		
		Suppose $ Empty(Q) = True $, with $ Empty(\cdot) $ as in Lemma \ref{lem.FK-CZ}(B). We set
		\begin{equation*}
			\xqs := center(Q).
		\end{equation*}
		It is clear that $ \xqs \in Q $ and \eqref{eq.xqs} holds.
		
		Suppose $ Empty(Q) = False $. Let $ x_0 := \erep(Q) \in E \cap 25Q $. 
		
		Suppose $ x_0 \notin 5Q $, then $ E \cap 5Q = \void$ by Lemma \ref{lem.FK-CZ}(C). Again, we set
		\begin{equation*}
			\xqs:= center(Q).
		\end{equation*}
		It is clear that $ \xqs \in Q $ and \eqref{eq.xqs} holds.

		Suppose $ x_0 \in 5Q $. This means that $ Q \in \Ls $ with $ \Ls $ as in \eqref{eq.Ls-def}. Let $ u_Q $ be as in Lemma \ref{lem.uQ}.

		By Lemma \ref{lem.CZ}(B), we have $ E \cap 5Q \subset \set{(t,\phi(t)):t \in \R} $ up to the rotation $ u_Q\mapsto e_2 $, and the function $ \phi $ satisfies $ \abs{\ddtm\phi(t)} \leq CA_1^{-1}\dq^{1-m} $ for $ m = 1,2 $, with $ A_1 $ as in Definition \ref{def.CZ}. Therefore, by the defining property of $ u_Q $ in Lemma \ref{lem.uQ}, we have 
		\begin{equation*}
			E \cap 5Q \subset \set{y \in \rt: \abs{(y-x_0)\cdot u_Q} \leq CA_1^{-1}\abs{y-x_0}}=:Z(x_0).
		\end{equation*}
		
		Suppose $ \dist{center(Q)}{Z(x_0)} \geq \dq/1024 $. We set
		\begin{equation*}
			\xqs := center(Q).
		\end{equation*}
		In this case, it is clear that $ \xqs \in Q $ and \eqref{eq.xqs} holds.
		
		Suppose $ \dist{center(Q)}{Z(x_0)} < \dq/1024 $. We set 
		\begin{equation*}
			\xqs:= center(Q) + \frac{\dq}{4}\cdot  u_Q.
		\end{equation*}
		It is clear that $ \xqs \in Q $. For sufficiently large $ A_1 $, we also have $ \dist{\xqs}{Z(x_0)} \geq c\dq $ for some constant $ c $ depending only on $ A_1 $. Thus, \eqref{eq.xqs} holds.
		
		After one-time work using at most $ CN\log N $ operations and $ CN $ storage, the procedure $ Empty(Q) $ requires at most $ C\log N $ operations by Lemma \ref{lem.FK-CZ}(B); the procedure $ \erep(Q) $ requires at most $ C\log N $ operations by Lemma \ref{lem.FK-CZ}(C); computing the vector $ u_Q $ requires at most $ C\log N $ operations; and computing the distance between $ center(Q) $ and $ Z(x_0) $ is a routine linear algebra problem, and requires at most $ C $ operations. 
		
		Lemma \ref{lem.rep} is proved. 
	\end{proof}

	We now turn our attention to $ \Lsk $ as in \eqref{eq.Lsk-def}. 
	
	\begin{lemma}\label{lem.Lsk}
		Using at most $ CN\log N $ operations and $ CN $ storage, we can compute the list $ \Lsk $ as in \eqref{eq.Lsk-def}. 
	\end{lemma}

	\begin{proof}
		This is a direct application of Lemma \ref{lem.FK-CZ}(A) to each $ x \in E $. 
	\end{proof}

	The next lemma states that we can efficiently sort the data contained in squares in $ \Lsk $. 
	
	\begin{lemma}\label{lem.sort}
		Using at most $ CN\log N $ operations and $ CN $ storage, we can compute the following.
		
		For each $ Q \in \Lsk $ with $ \Lsk $ as in \eqref{eq.Lsk-def}, we can compute a sorted list of numbers 
		\begin{equation*}
			Proj_{u_Q^\perp}(E \cap (1+c_G)Q - \erep(Q)) \subset \R,
		\end{equation*}
		where $ u_Q^{\perp} $ is as in Lemma \ref{lem.uQ},  $Proj_{u_Q^\perp} $ is the orthogonal projection onto $ \R u_Q^{\perp} $, and $ \erep(Q) $ is as in Lemma \ref{lem.FK-CZ}(C).
	\end{lemma}

	\begin{proof}
		By the bounded intersection property in Lemma \ref{lem.CZ}(A), we have
		\begin{equation}
			\#(\Lsk) \leq CN.
			\label{eq.5.8.2}
		\end{equation}
		
		From the definitions of $ \Lsk $ and $ \Ls $ in \eqref{eq.Lsk-def} and \eqref{eq.Ls-def}, we see that $ \Lsk \subset \Ls $. Therefore, we can compute $ \erep(Q) $ and $ u_Q^\perp $ for each $ Q \in \Lsk $ using at most $ C\log N $ operations, by Lemma \ref{lem.FK-CZ}(B) and Lemma \ref{lem.uQ}.
		
		Recall from Lemma \ref{lem.Lsk} that we can compute the list $ \Lsk $ by computing each $ \Lambda(x) $ for $ x \in E $, with $ \Lambda(x) $ as in Lemma \ref{lem.FK-CZ}(A). During this procedure, we can store the information $ (1+c_G)Q \ni x $ for $ Q \in \Lambda(x) $.
		
		By the bounded intersection property in Lemma \ref{lem.CZ}(A), we have
		\begin{equation}
			\sum_{Q \in \Lsk}\#(E \cap (1+c_G)Q) \leq CN.
			\label{eq.5.8.1}
		\end{equation}
		By Lemma \ref{lem.FK-CZ}(A) and \eqref{eq.5.8.1}, we can compute the list
		\begin{equation*}
			\set{E \cap (1+c_G)Q : Q \in \Lsk}
		\end{equation*}
		using at most $ CN\log N $ operations and $ CN $ storage. Then, by Lemma \ref{lem.FK-CZ}(C), Lemma \ref{lem.uQ}, and \eqref{eq.5.8.2}, we can compute the {\em unsorted} list
		\begin{equation}
			Proj_{u_Q^\perp}(E \cap (1+c_G)Q - \erep(Q))
			\label{eq.list}
		\end{equation}
		for each $ Q \in \Lsk $
		using at most $ CN\log N $ operations and $ CN $ storage. 
		
		For each $ Q \in \Lsk $, we can sort the list $ Proj_{u_Q^\perp}(E \cap (1+c_G)Q - \erep(Q)) $ using at most $ CN_Q\log N_Q $ operations, where $ N_Q := \#(E \cap (1+c_G)Q) $. By \eqref{eq.5.8.1}, we can sort the all the lists of the form \eqref{eq.list} associated with each $ Q \in \Lsk $ using at most $ CN\log N $ operations. 
		
		Lemma \ref{lem.sort} is proved.

	\end{proof}

	\subsection{Local clusters}

	The next lemma shows how to relay local information to the point $ \xqs $.

	\begin{lemma}\label{lem.SQ}
		Let $ Q \in \Ls $. Let $ \xqs $ be as in Lemma \ref{lem.rep}. Let $ x \in E \cap 5Q $. Let $ \A(x) $ be as in Lemma \ref{lem.FK-palp}. Let $ S(\A(x)) $ be as in \eqref{eq.S(A)-def}. Then 
		\begin{equation}
			\sigma(\xqs,S(\A(x))) \subset C\cdot \sk(\xqs,16).
			\label{eq.5.6.0}
		\end{equation}
		
	\end{lemma}
	
	\begin{proof}
		\newcommand{\sax}{S(\A(x))}
		\renewcommand{\set}[1]{\{#1\}}
		Fix $ x $ as in the hypothesis. By our choice of $ \xqs $ in Lemma \ref{lem.rep}, we have
		\begin{equation}
			\abs{\xqs - x} \geq C\dq. 
			\label{eq.5.10.1}
		\end{equation}
		
		Let $ P_0 \in \sigma(\xqs, S(\A(x))) $.
		By the definition of $ \sigma $, there exists $ \phi\in \ct(\rt) $ with $ \norm{\phi}_{\ct(\rt)} \leq 1 $, $ \phi|_{S(\A(x))} = 0 $, and $ \jet_\xqs\phi = P_0 $. Set $ P := \jet_x\phi $. Then
		\begin{equation*}
			P \in \sigma(x,\sax). 
		\end{equation*}
		Since $ x \in E $, by Lemma \ref{lem.sigma-main}, we have
		\begin{equation*}
			P \in \sk(x,16). 
		\end{equation*}
		
		Let $ S \subset E $ with $ \#(S) \leq 16 $. By the definition of $ \sk $ in \eqref{eq.sigma-def} and Taylor's theorem, there exists a Whitney field $ \vec{P} = (P,(P^y)_{y \in S}) \in \wt(S\cup\set{x}) $, with $ \norm{\vec{P}}_{\wt(S\cup\set{x})} \leq C $ and $ P^y(y) = 0 $ for $ y \in S $. 
		
		Consider another Whitney field $ \vec{P}_0 = (P_0, (P^y)_{y \in S}) \in \wt(S\cup\set{\xqs}) $ defined by replacing $ P $ by $ P_0 $ in $ \vec{P} $. By the classical Whitney Extension Theorem for finite sets, it suffices to show that $ \vec{P}_0 $ satisfies 
		\begin{equation}\label{eq.5.10.2}
			P^y(y) = 0\for y \in S, \text{ and }
		\end{equation}
		\begin{equation}\label{eq.5.10.3}
			\norm{\vec{P}_0}_{\wt(S\cup\set{\xqs})} \leq C.
		\end{equation}
		
		Note that \eqref{eq.5.10.2} is obvious by construction.
		
		We turn to \eqref{eq.5.10.3}.
		
		Since $ P_0 = \jet_\xqs\phi $ and $ P = \jet_x\phi $, Taylor's theorem implies 
		\begin{equation}
			\abs{\d^\alpha (P - P_0)(\xqs)}, \abs{\d^\alpha(P - P_0)(x)} \leq C\abs{x - \xqs}^{2-\abs{\alpha}}
			\for\abs{\alpha} \leq 1.
			\label{eq.5.10.4}
		\end{equation}
		Since the Whitney field $ \vec{P} = (P,(P^y)_{y \in S}) $ satisfies $ \norm{\vec{P}}_{\wt(S\cup\set{x})} \leq C $, we have
		\begin{equation}
			\norm{(P^y)_{y \in S}}_{\wt(S)} \leq C, 
			\label{eq.5.10.5}
		\end{equation}
		and
		\begin{equation}
			\abs{\d^\alpha(P - P^y)(x)}, \abs{\d^\alpha(P - P^y)(y)} \leq C\abs{x - y}^{2-\abs{\alpha}}
			\for\abs{\alpha} \leq 2, y \in S.
			\label{eq.5.10.6}
		\end{equation}
		Applying the triangle inequality to \eqref{eq.5.10.4} and \eqref{eq.5.10.6}, and using \eqref{eq.5.10.1}, we see that
		\begin{equation}
			\abs{\d^\alpha (P_0 - P^y)(\xqs)}, \abs{\d^\alpha (P_0 - P^y)(y)} \leq C\abs{\xqs - y}^{2-\abs{\alpha}}
			\for\abs{\alpha} \leq 1.
			\label{eq.5.10.7}
		\end{equation}
		Moreover, since $ P_0 \in \sigma(\xqs,\sax) $, we have
		\begin{equation}
			\abs{\d^\alpha P_0(\xqs)} \leq 1
			\for\abs{\alpha} \leq 1.
			\label{eq.5.10.8}
		\end{equation}
		Then, \eqref{eq.5.10.3} follows from \eqref{eq.5.10.5}, \eqref{eq.5.10.7}, and \eqref{eq.5.10.8}. 
		
		Lemma \ref{lem.SQ} is proved. 
		
	\end{proof}

	\newcommand{\ssq}{{S^\sharp(Q)}}
	\newcommand{\qqs}{{\Q^\sharp}}
	\newcommand{\mqs}{{\M^\sharp}}
	
	Let $ Q \in \Ls $ with $ \Ls $ as in \eqref{eq.Ls-def}. Let $ \A(x), x \in E $ be as in Lemma \ref{lem.FK-palp}. Let $ S(\A(x)) $ be as in \eqref{eq.S(A)-def}. Let $ \erep(Q) $ be as in Lemma \ref{lem.FK-CZ}(C). Let $ \xqs $ be as in Lemma \ref{lem.rep}. We set
	\begin{equation}
		\ssq := S(\A(\erep(Q))) \cup \set{\erep(Q)} \cup \{\xqs\}. 
		\label{eq.ssq-def}
	\end{equation}
	Note that $ \xqs $ is not a point in $ E $.

	\subsection{Transition jets}
	\label{section:transition jet}

	In this section, we want construct a map $ T_Q:\ctp(E)\times \pos \to \P $ of bounded depth, such that $ T_Q(f,M) \in \Gk(\xqs,16,CM,f) $ for all $ (f,M) \in \ctp(E)\times \pos $ with $ \norm{f}_{\ctp(E)} \leq M $. We will explain the importance of $ \Gk(\xqs,16,CM,f) $ in Remark \ref{rem.16} towards the end of the section.

	Let $ S \subset E $. As in \eqref{eq.Q-def} and \eqref{eq.M-def}, we consider the following functions, depending on the choice of $ S $:
	\begin{equation}
		\begin{split}
			\qqs: \wt(S) &\to \pos\\
			\vec{P} = (P^x)_{x \in S} &\mapsto \sum_{x \in S}\abs{\d^\alpha P^x(x)} + \sum_{\substack{x, y \in \ssq\\x \neq y\\\abs{\alpha}\leq 1}} { \frac{\abs{\d^\alpha (P^x - P^y)(x)}}{\abs{x - y}^{2-\abs{\alpha}}} }
		\end{split}\,,
		\label{eq.QQs-def}
	\end{equation}
	and
	\begin{equation}
		\begin{split}
			\mqs: \wtp(S) &\to [0,\infty]\\
			(P^x)_{x \in S} &\mapsto \begin{cases}
				\sum\limits_{x \in S} \frac{\abs{\grad P^x}^2}{P^x(x)} &\text{ if $ P^x(x) \geq 0 $ for each $ x \in S $}\\
				\infty &\text{ if there exists $ x \in \ssq $ such that $ P^x(x) < 0 $}
			\end{cases}\,.
		\end{split}
		\label{eq.MQs-def}
	\end{equation}
	We adopt the conventions that $ \frac{0}{0}= 0 $ and $ \frac{a}{0} = \infty $ for $ a > 0 $.

	For the rest of the section, we fix $ Q \in \Ls $, with $ \Ls $ as in \eqref{eq.Ls-def}. Let $ \xqs $ be as in Lemma \ref{lem.rep}. Let $ \ssq $ be as in \eqref{eq.ssq-def}. Recall from \eqref{eq.ssq-def} that $ \erep(Q) \in \ssq $, with $ \erep $ as in Lemma \ref{lem.FK-CZ}(C).

	\renewcommand{\Af}{\mathbb{A}_f}
	Let $ f \in \ctp(E) $ be given. We define
	\begin{equation}
		\begin{split}
			\Af^0 &:= \set{\vp \in \wtp(\ssq): \begin{matrix*}[l]
					(\vp,\xqs) \equiv 0 \text{ and }\\
					(\vp,x)(x) = f(x) \for x \in \ssq\cap E
			\end{matrix*}}\text{, and }\\
			\Af^1 &:= \set{\vp \in \wtp(\ssq\cap E): \begin{matrix}
					(\vp,x)(x) = f(x) \for x \in \ssq\cap E
			\end{matrix}}\,.
		\end{split}
		\label{eq.Af-def}
	\end{equation}
	We note that that $ \Af^0 $ and $ \Af^1 $ are affine subspaces of $ \wt(\ssq) $ and $ \wt(\ssq\cap E) $, respectively. They depend only on $ f|_{\ssq\cap E} $. 
	
	Consider the following minimization problems.
	\begin{enumerate}[(M1)]
		\setcounter{enumi}{-1} 
		
		\item\label{M0} Let $ S = \ssq $ in \eqref{eq.QQs-def} and \eqref{eq.MQs-def}. Minimize $ \qqs+\mqs $ over $ \Af^0 $. 
		
		\item\label{M1} Let $ S = \ssq \cap E $ in \eqref{eq.QQs-def} and \eqref{eq.MQs-def}. Minimize $ \qqs+\mqs $ over $ \Af^1 $,

	\end{enumerate}

	For $\star = 0,1 $, we say a Whitney field $ \vp \in \Af^\star $ is an \ul{approximate minimizer} of (M$ \star $) if
	\begin{itemize}
		\item $ (\qqs+\mqs)(\vp) \leq C\cdot \inf\set{(\qqs+\mqs)(\vp') : \vp' \in \Af^\star} $ for some universal constant $ C $.
	\end{itemize}
	
	\begin{remark}\label{rem.approx-min}
		Recall from Section \ref{sect:lasso} that both (M0) and (M1) can be reformulated as convex quadratic programming problems with affine constraint, and are efficiently solvable\cite{BV-CO}. Thus, we can solve for an approximate minimizer of (M$ \star $), $ \star = 0,1 $, using at most $ C $ operations, since $ \#(\ssq) $ is universally bounded. We call the approximate minimizers for (M0) and (M1) obtained this way $ \vp_0^\sharp $ and $ \vp_1^\sharp $. Note that $ \vp_0^\sharp $ and $ \vp_1^\sharp $, respectively, are uniquely determined by $ \Af^0 $ and $ \Af^1 $.  
	\end{remark}

	\begin{lemma}\label{lem.M1}
		Let $ Q \in \Ls $. Let $ \xqs $ be as in Lemma \ref{lem.rep}. Let $ (f,M) \in \ctp(E)\times\pos $ with $ \norm{f}_{\ctp(E)} \leq M $. Let $ \vec{P} = (P^x)_{x \in \ssq \cap E} $ be an approximate minimizer of (\hyperref[M1]{M1}) above. Let $ P^{\erep(Q)} $ be the polynomial associated with the point $ \erep(Q) $, i.e., $ P^{\erep(Q)} = (\vp,\erep(Q)) $, with $ \erep $ as in Lemma \ref{lem.FK-CZ}(C). Let $ T_w^{\erep(Q)} $ be the Whitney extension operator associated with the singleton $ \set{\erep(Q)} $ as in Lemma \ref{lem.WT}(B). 
		Then 
		\begin{equation*}
			\jet_{\xqs} \circ T_w^{\erep(Q)} (P^{\erep(Q)}) \in \G(\xqs,\ssq\cap E, CM, f).
		\end{equation*}
		
	\end{lemma}
	
	\begin{proof}
		
		Let $ \vec{P} $ be as in the hypothesis. Let $ P_1 := \jet_{\xqs} \circ T_w^{\erep(Q)} (P^{\erep(Q)}) $. We adjoin $ P_1 $ to $ \vp $ to form 
		\begin{equation*}
			\vp_1 := \brac{P_1, (P^x)_{x \in \ssq \cap E}} \in \wt(\ssq) .
		\end{equation*}
		Thanks to Lemma \ref{lem.WT}, it suffices to show that $ \vp_1 \in \wtp(\ssq) $ and $ \norm{\vp_1}_{\wtp(\ssq)} \leq CM $. 
		
		By Lemma \ref{lem.WT}(B), we see that $ T_w^{\erep(Q)} (P^{\erep(Q)}) \in \ctp(\rt) $ with norm $ \norm{T_w^{\erep(Q)} (P^{\erep(Q)}) }_{\ct(\rt)} \leq CM $. Therefore, 
		\begin{equation}
			\abs{\d^\alpha P_1(\xqs)} \leq CM\text{ for }\abs{\alpha} \leq 1,
			\text{ and }
			\abs{\grad P_1} \leq \sqrt{CMP_1(\xqs)}.
			\label{eq.5.30}
		\end{equation}
		Thus, $ \vp_1 \in \wtp(\ssq) $.

		Since $ \vp $ is an approximate minimizer of (\hyperref[M1]{M1}) and $ \norm{f}_{\ctp(E)} \leq M $, we have
		\begin{equation}
			\norm{\vp}_{\wtp(\ssq \cap E)}\leq CM.
			\label{eq.5.31}
		\end{equation}
		For $ x \in \ssq\cap E $, we have
		\begin{equation*}
			\begin{split}
				\abs{\d^\alpha(P^x - P_1)(x)} 
				&\leq \abs{\d^\alpha (P^x - P^{\erep(Q)})(x)} + \abs{\d^\alpha(P^{\erep(Q)} - \jet_{\xqs} \circ T_w^{\erep(Q)} (P^{\erep(Q)}))(x)}.
			\end{split}
			\label{eq.5.32}
		\end{equation*}
		Using \eqref{eq.5.31} to estimate the first term and Taylor's theorem to estimate the second, we have
		\begin{equation}\label{eq.5.33}
			\abs{\d^\alpha(P^x - P_1)(x)} \leq CM\brac{
				\abs{x-\erep(Q)} +
				\abs{\xqs - \erep(Q)}
			}^{2-\abs{\alpha}}
			\leq CM\abs{x-\xqs}^{2-\abs{\alpha}}.
		\end{equation}
		For the last inequality, we use the fact that $ \dist{\xqs}{E} \geq c\dq $, thanks to Lemma \ref{lem.rep}. 
		
		Applying Taylor's theorem to \eqref{eq.5.33}, we have
		\begin{equation}\label{eq.5.34}
			\abs{\d^\alpha(P^x - P_1)(\xqs)} \leq CM\abs{x-\xqs}^{2-\abs{\alpha}}.  
		\end{equation}
		
		Combining \eqref{eq.5.30}--\eqref{eq.5.34}, we see that $ \norm{\vp_1}_{\wtp(\ssq)} \leq CM $. Lemma \ref{lem.M1} is proved.

	\end{proof}

	\begin{definition}\label{def.TQ}
		Let $ Q \in \Ls $. Let $ \xqs $ be as in Lemma \ref{lem.rep}. We define
		\begin{equation*}
			T_Q: \ctp(E)\times\pos \to \P
		\end{equation*}
		by the following rule. Let $ (f,M) \in \ctp(E) \times \pos $ be given, and let (M0) and (M1) be as above. Let $ \vp_0^\sharp $ and $ \vp_1^\sharp $ be as in Remark \ref{rem.approx-min}. 
		\begin{enumerate}[(TQ-1)]
			\setcounter{enumi}{-1} 
			\item Suppose $ \vp_0^\sharp $ satisfies $ (\qqs+\mqs)(\vp_0^\sharp) \leq C_TM $, for some large universal constant $ C_T $. Then we set $ T_Q(f,M) \equiv 0 $.
			\item Otherwise, we set $ T_Q(f,M) := \jet_{\xqs}\circ T_w^{\erep(Q)}\brac{P_1} $. Here, $ P_1 $ is the polynomial in $ \vp_1^\sharp $ associated with the point $ \erep(Q) $, i.e., $ P_1:= (\vp_1^\sharp,\erep(Q)) $; and $ T_w^{\erep(Q)} $ is the Whitney extension operator associated with the singleton $ \set{\erep(Q)} $ as in Lemma \ref{lem.WT}(B). 
		\end{enumerate}
	\end{definition}
	
	It is clear that $ T_Q $ has bounded depth, since $ \vp^\sharp_0 $ and $  \vp^\sharp_1 $ depend only on $ f|_{\ssq\cap E} $. 
	
	\begin{remark}\label{rem.TQ}
		Given $ Q \in \Ls $ with $ \Ls $ as in \eqref{eq.Ls-def}, $ \xqs $ as in Lemma \ref{lem.rep}, $ \ssq $ as in \eqref{eq.ssq-def}, and $ (f,M) \in \ctp(E) \times \pos $ with $ \norm{f}_{\ctp(E)} \leq M $, computing $ T_Q(f,M) $ from the data above amounts to solving for approximate minimizers of (M0) and (M1). Thus, by Remark \ref{rem.approx-min}, we can compute $ T_Q(f,M) $ from the data above using at most $ C $ operations.
	\end{remark}
	
	Recall the following perturbation lemma from \cite{JL20}.
	
	\begin{lemma}[variant of Lemmas 5.7 and 7.3 of \cite{JL20}]\label{lem.perturb}
		Let $ E \subset \rt $ be finite. Let $ Q \in \Ls $. Let $ \xqs $ be as in Lemma \ref{lem.rep}. Let $ f \in \ctp(E) $ be given. Suppose $ \Gk(\xqs,16,M,f) \neq \void $. The following are true. 
		\begin{enumerate}[(A)]
			\item There exists a number $ B_0 > 0 $ exceeding a large universal constant such that the following holds. Suppose $ f(x) \geq B_0M\dq^2 $ for each $ x \in E \cap 5Q $. Then 
			\begin{equation*}
				\eqindent
				\Gk(\xqs,16,M,f) + M\cdot \sk(\xqs,16)\subset \Gk(\xqs,16,CM,f),
			\end{equation*}
			for some universal constant $ C $.
			\item Let $ A > 0 $. Suppose $ f(x) \leq AM\dq^2 $ for some $ x \in E \cap 5Q $. Then
			\begin{equation*}
				0 \in \Gk(\xqs,16,A'M,f).
			\end{equation*}
			Here, $ A' $ depends only on $ A $. 
		\end{enumerate}
	\end{lemma}
	
	The main lemma of this section is the following.
	
	\begin{lemma}\label{lem.TQ}
		Let $ Q \in \Ls $ with $ \Ls $ as in \eqref{eq.Ls-def}. Let $ \xqs $ be as in Lemma \ref{lem.rep}. Let $ T_Q $ be as in Definition \ref{def.TQ}. Let $ (f,M) \in \ctp(E) \times \pos $ with $ \norm{f}_{\ctp(E)} \leq M $. Then
		\begin{equation*}
			T_Q(f,M) \in \Gk(\xqs,16,CM,f).
		\end{equation*}
	\end{lemma}

	\begin{proof}
		Since $ \norm{f}_{\ctp(E)} \leq M $, we have $ \Gk(\xqs,16,CM,f) \neq \void $. Therefore, the hypotheses of Lemma \ref{lem.perturb} are satisfied. 
		
		Recall Definition \ref{def.TQ}. 
		\newcommand{\GM}{\G(\xqs,\ssq\cap E,CM,f)}
		
		Suppose $ T_Q(f,M) $ is defined in terms of (TQ-0). 
		
		By Lemma \ref{lem.WT}, there exists $ F \in \ctp(\rt) $ with $ \norm{F}_{\ct(\rt)} \leq CM $, $ F|_{\ssq \cap E} = f $, and $ \jet_{\xqs} F \equiv 0 $. Recall from Lemma \ref{lem.FK-CZ}(C) and \eqref{eq.ssq-def} that $ \erep(Q) \in \ssq \cap 5Q $. Therefore, by Taylor's theorem, we have
		\begin{equation*}
			f(\erep(Q)) = F(\erep(Q)) \leq CM\dq^2. 
		\end{equation*}
		By Lemma \ref{lem.perturb}(B), we have $ T_Q(f,M) \equiv 0 \in \Gk(\xqs,16,CM,f) $. 
		
		Suppose $ T_Q(f,M) $ is defined in terms of (TQ-1). 
		
		For sufficiently large $ C_T $, Taylor's theorem implies, with $ B_0 $ as in Lemma \ref{lem.perturb},
		\begin{equation*}
			f(x) \geq B_0 M\dq^2
			\for x \in E \cap 5Q.
		\end{equation*}
		Thus, the hypothesis of Lemma \ref{lem.perturb}(A) is satisfied.  
		
		Since $ \norm{f}_{\ctp(E)} \leq M $, there exists 
		\begin{equation*}
			\hat{F} \in \ctp(\rt) \text{ with }
			\norm{\hat{F}}_{\ct(\rt)} \leq CM,\,
			\hat{F}|_E = f, \text{ and }
			\jet_{\xqs} \hat{F} \in \G(\xqs,E,CM,f).
		\end{equation*}

		By Lemma \ref{lem.M1}, we have
		\begin{equation*}
			T_Q(f,M) \in \GM.
		\end{equation*}
		
		Therefore, by Lemma \ref{lem.SQ}, the definition of $ \ssq $ in \eqref{eq.ssq-def}, and the definition of $ \sigma $ in \eqref{eq.sigma-def}, we have
		\begin{equation*}
			\jet_{\xqs}\hat{F} - T_Q(f,M) \in CM\cdot \sigma(\xqs,\ssq\cap E) \subset C' M\cdot \sk(\xqs,16).
		\end{equation*}
		
		Thus, by Lemma \ref{lem.perturb}(A) and the trivial inclusion $ \G(\xqs,E,M,f) \subset \Gk(\xqs,16,M,f) $, we have
		\begin{equation*}
			\begin{split}
				T_Q(f,M) &\in \jet_{\xqs}\hat{F} + CM\cdot\sk(\xqs,16) \\
				&\subset \Gk(\xqs,16,CM,f) + CM\sk(\xqs,16) \\
				&\subset \Gk(\xqs,16,C' M,f).
			\end{split}
		\end{equation*}

		Lemma \ref{lem.TQ} is proved.

	\end{proof}

	\begin{remark}\label{rem.16}
		We will not use Lemma \ref{lem.TQ} explicitly in this paper. However, jets in $ \Gk(\xqs,16,M,f) $ are crucial for the following reason: 
		\LAQ{eq.trans-jet}{(Lemma 5.3 of \cite{JL20}) Suppose $ Q, Q' \in \Lz $, $ \xqs $ and $ x_{Q'}^\sharp $ as in Lemma \ref{lem.rep}, $ P \in \Gk(\xqs,16,M,f) $ and $ P' \in \Gk(x_{Q'}^\sharp,16,M,f) $, then
			\begin{equation*}
				\abs{\d^\alpha(P - P')(\xqs)}, \abs{\d^\alpha(P - P')(x_{Q'}^\sharp)} \leq CM\brac{\dq + \delta_{Q'} + \abs{\xqs - x_{Q'}^\sharp}}^{2-\abs{\alpha}}
				\for \abs{\alpha} \leq 1.
		\end{equation*}}
		We can then use \eqref{eq.trans-jet} to control the derivatives when we patch together local extensions. See the proof of Theorem 1 in \cite{JL20-Ext}.
	\end{remark}

	\subsection{One-dimensional algorithms}

	\newcommand{\bxi}{\overline{\Xi}}
	\newcommand{\bP}{\overline{\P}}
	\newcommand{\bjet}{\overline{\jet}}
	
	We write $ \bP, \bP^+ $, respectively, to denote the collections of single-variable polynomials of degree no greater than one, two. We write $ \bjet_t, \bjet_t^+ $, respectively, to denote the one-jet, two-jet, of a single variable function at $ t \in \R $.

	We recall the following results proven in \cite{JL20}.
	
	\newcommand{\Eb}{\overline{\E}}
	\newcommand{\Ebpm}{\overline{\E}_{\pm}}
	\newcommand{\bez}{\overline{E}_0}

	\begin{subtheorem}{theorem}
		
		\renewcommand{\ddtm}{\d^m}
		
		\begin{theorem}\label{thm.bd-1d}
			Let $ \bez \subset \R $ be a finite set with $ \#(\bez) = N_0 $. We think of $ \ctp(\bez) \approx \pos^{N_0} $. Then there exists a collection of maps $ \set{\bxi_{+}^t: t \in \R} $, where $ \bxi_{+}^t:\ctp(E) \to \bP^+ $ for each $ t \in \R $, such that the following hold.
			\begin{enumerate}[(A)]
				\item There exists a universal constant $ D_0 $ such that for each $ t \in \R $, the map $ \bxi_+^t: \ctp(\bez) \to \bP^+ $ is of depth $ D_0 $. 
				\item Let $ f \in \ctp(\bez) $ be given. Then there exists a function $ F \in \ctp(\R) $ such that 
				\begin{equation*}
					\bjet_t^+F = \bxi_+^t(f)
					\text{ for all } t \in \R,
					\enskip
					\norm{F}_{\ct(\R)} \leq C\norm{f}_{\ctp(\bez)}, \text{ and }
					F(t) = f(t)
					\for t \in E.
				\end{equation*}
				\item There is an algorithm, that takes the given data, performs one-time work, and then responds to queries. 
				
				A query consists of a point $ t \in \R $, and the response to the query is the depth-$ D_0 $ map $ \bxi_+^t $, given in its efficient representation.
				
				The one-time work takes $ CN\log N $ operations and $ CN $ storage. The time to answer a query is $ C\log N $.
			\end{enumerate}

		\end{theorem}

		\begin{theorem}\label{thm.bd-1dpm}
			Let $ \bez \subset \R $ be a finite set with $ \#(\bez) = N_0 $. We think of $ \ct(\bez) \approx \R^{N_0} $. Then there exists a collection of maps $ \set{\bxi_{\pm}^t: t \in \R} $, where $ \bxi_{\pm}^t:\ct(E) \to \bP^+ $ for each $ t \in \R $, such that the following hold.
			\begin{enumerate}[(A)]
				\item There exists a universal constant $ D_0 $ such that for each $ t \in \R $, the map $ \bxi_\pm^t: \ct(\bez) \to \bP^+ $ is linear and of depth $ D_0 $. 
				\item Let $ f \in \ct(\bez) $ be given. Then there exists a function $ F \in \ct(\R) $ such that 
				\begin{equation*}
					\bjet_t^+F = \bxi_\pm^t(f)
					\text{ for all } t \in \R,
					\enskip
					\norm{F}_{\ct(\R)} \leq C\norm{f}_{\ct(\bez)}, \text{ and }
					F(t) = f(t)
					\for t \in E.
				\end{equation*}
				\item There is an algorithm, that takes the given data, performs one-time work, and then responds to queries. 
				
				A query consists of a point $ t \in \R $, and the response to the query is the depth-$ D_0 $ map $ \bxi_{\pm}^t $, given its efficient representation.
				
				The one-time work takes $ CN\log N $ operations and $ CN $ storage. The time to answer a query is $ C\log N $.
			\end{enumerate}
			
		\end{theorem}

	\end{subtheorem}

	The explanation for Theorems \ref{thm.bd-1d} and \ref{thm.bd-1dpm} without the complexity statements were given in \cite{JL20}. We repeat the explanations for completeness, and further elaborate on the complexity.
	
	Using at most $ CN_0\log N_0 $ operations and $ CN_0 $ storage, we can sort 
	\begin{equation*}
		\bez = \set{t_1, \cdots, t_{N_0}} 
		\text{ with }
		t_1 < \cdots < t_{N_0}.
	\end{equation*}
	
	Let us begin with Theorem \ref{thm.bd-1d}. 
	
	Suppose $ \#(\bez) \leq 3 $. Let $ \overline{\Q} $ and $ \overline{\M} $ be as in \eqref{eq.Q-def} and \eqref{eq.M-def}, but with $ \bP $ instead of $ \P $. Let $ \overline{f} : \bez \to \pos $. Let $ \vec{P}_0 $ be a section of $ \bez \times \bP $ (i.e., a Whitney field in one-dimension) that minimizes $ ({\overline{\Q}} + \overline{M}) $ subject to the constraint $ (\vec{P}_0,t)(t) = f(t) $ for $ t \in \bez $ (see Section \ref{sect:lasso}).
	Let $ \overline{T}_W $ be the one-dimensional counterpart of the operator in Lemma \ref{lem.WT}(B). Then $ \overline{F}:= \overline{T}_W(\vec{P}_0) \in \ct(\R) $ with $ \overline{F}(t) = f(t) $ and $ \overline{F}(t) \geq 0 $ on $ \R $, thanks to Lemma \ref{lem.WT}(B). By the one-dimensional counterpart of Lemma \ref{lem.Q+M}, we have $ \norm{\overline{F}}_{\ct(\R)} \leq C\norm{\overline{f}}_{\ctp(\bez)} $. Thus, we have constructed a bounded nonnegative extension operator $ \Eb:\ctp(\bez) \to \ctp(\R) $ if $ \#(\bez) \leq 3 $. We can simply take the map $ \overline{\Xi}_t(\cdot) $ in Theorem \ref{thm.bd-1d}(B) to be $ \bjet_t^{+}\circ\Eb(\cdot) $. 
	
	We have shown in Theorem 2.A. of \cite{JL20} that there exists a bounded nonnegative extension operator $ \Eb:\ctp(\bez) \to \ctp(\R) $ of bounded-depth in the form
	\begin{equation}
		\Eb(f)(t) = \sum_{i = 1}^{N_0 -2}\theta_i(t)\cdot \Eb_i(f)(t),
		\label{eq.Eb-def}
	\end{equation}
	where 
	\begin{itemize}
		\item $ \bez^{(i)} = \set{t_i, t_{i+1}, t_{i+2}} $,
		\item $ \Eb^i(\cdot): \ctp(\bez^{(i)}) \to \ctp(\R) $ is the bounded nonnegative extension operator constructed in the previous step, and
		\item $ \theta_1, \theta_2, \cdots, \theta_{N_0 - 3}, \theta_{N_0-2} $ form a nonnegative $ C^2 $ partition of unity subordinate to the cover {$ (-\infty, t_3), (t_2, t_4), \cdots, (t_{N_0 -3}, t_{N_0-1}), (t_{N_0-2},\infty) $}, such that\begin{equation*}
			\abs{\ddtm \theta_i(\hat{t})} \leq \begin{cases}
				C\abs{t_{i+1}-t_i}^{-m} \text{ if }\hat{t} \in (t_i, t_{i+1})\\
				C\abs{t_{i+2}-t_{i+1}}^{-m}\text{ if } \hat{t} \in (t_{i+1}, t_{i+2})
			\end{cases}\,,
			\for i = 1, \cdots, N_0 -2.
		\end{equation*}
	\end{itemize}
	
	Given $ t \in \R $ and $ i \in \set{1, \cdots, N_0-2} $, we can compute $ \bjet_t^+\theta_i $ using at most $ C\log N_0 $ operations. 
	
	Let $ t \in \R $ be given. Note that $ t $ is supported by at most two of the $ \theta_i $'s. In $ C\log N_0 $ operations, we can find all $ i', i'' \in \set{1, \cdots, N_0-2} $ (possibly $ i' = i'' $) such that $ t \in \supp{\theta_{i'}}\cup\supp{\theta_{i''}} $. It is a standard search algorithm and requires at most $ C\log N_0 $ operations, since $ \bez $ has been sorted. Finally, we simply set
	\begin{equation*}
		\overline{\Xi}_t(\cdot) := \bjet_t^+\circ \brac{\sum_{i \in\set{i', i''}}\theta_i\cdot \Eb_i(\cdot)}.
	\end{equation*}
	
	It is clear from construction that $ \Xi^t_+(\cdot) $ depends only on $ f|_{\overline{S}(t)} $, where
	\begin{equation}
		\overline{S}(t):=\begin{cases}
			\bez &\text{ if } \#(\bez) \leq 3\\
			\text{three closest points in $\bez$ closest to $t$} &\text{ if } \#(\bez) > 3 \text{ and } t \notin [t_1, t_{N_0}]\\
			\set{t_1, t_2, t_3} &\text{ if } t \in [t_1, t_2]\\
			\set{t_{N_0-2}, t_{N_0 - 1}, t_{N_0}} &\text{ if } t \in [t_{N_0 - 1}, t_{N_0}]\\
			\set{t_1', t_2', t_3', t_4'} \subset \bez \text{ with } t_1' < t_2' \leq t \leq t_3' < t_4' &\text{ otherwise}
		\end{cases} \,.
		\label{eq.bst}
	\end{equation}
	Theorem \ref{thm.bd-1d}(A) then follows.
	
	Theorem \ref{thm.bd-1d}(B) follows from the fact that the operator $ \Eb $ in \eqref{eq.Eb-def} is a bounded nonnegative extension operator on $ \ctp(\bez) $. 
	
	Theorem \ref{thm.bd-1d}(C) follows from the discussions above on complexity.
	
	We have finished explaining Theorem \ref{thm.bd-1d}.
	
	The explanation for Theorem \ref{thm.bd-1dpm} is almost identical with some simplification, which we explain below.
	
	When constructing a bounded extension operator for $ \ct(\bez) $ with $ \#(\bez) \leq 3 $, we use 
	\begin{itemize}
		\item the natural quadratic form associated with $ \wt(\bez) $ instead of $ ({\overline{\Q}} + \overline{M}) $; and
		\item the classical Whitney extension operator instead of $ T_w $ in Lemma \ref{lem.WT}(B).
	\end{itemize}
	See \cite{F05-L,FK09-Data-1,FK09-Data-2} for details and further discussion on linear extension operators without the nonnegative constraint.

	This concludes the explanation for Theorem \ref{thm.bd-1dpm}.

	\subsection{Local extension problem}
	
	\newcommand{\xxq}{{\Xi_{x,Q}}}
	
	The main lemma of the section is the following.

	\begin{lemma}\label{lem.loc-ext}
		Let $ Q \in \Lsk $ with $ \Lsk $ as in \eqref{eq.Lsk-def}. There exists a collection of maps $ \set{\xxq : x \in (1+c_G)Q} $ where $ \xxq : \ctp(E) \times [0,\infty) \to \P^+ $ for each $ x \in (1+c_G)Q $, such that the following hold.
		\begin{enumerate}[(A)]
			\item There exists a universal constant $ D $ such that for each $ x \in (1+c_G)Q $, the map $ \xxq(\cdot\,,\cdot) : \ctp(E) \to \P^+ $ is of depth $ D $. 
			
			\item Suppose we are given $ (f,M) \in \ctp(E) \times \pos $ with $ \norm{f}_{\ctp(E)} \leq M $. Then there exists a function $ F_{Q} \in \ctp((1+c_G)Q) $ such that
			\begin{enumerate}[(\text{B}1)]
				\item $ \jet^+_x F_{Q} = \xxq(f,M) \text{ for all } x \in (1+c_G)Q $;
				\item $ \norm{F_{Q}}_{\ct((1+c_G)Q)} \leq CM $;
				\item $ F_{Q}(x) = f(x) \for
				x \in E\cap (1+c_G)Q $; and
				\item $ \jet_{\xqs}F_{Q} \in \Gk(\xqs,16,CM,f) $, with $ \xqs $ as in Lemma \ref{lem.rep} and $ \Gk $ as in \eqref{eq.G-def}. 
			\end{enumerate}

			\item There is an algorithm, that takes $ (E,f,M,Q) $ as input, performs one-time work, and then responds to queries.
			
			A query consists of a point $ x \in (1+c_G)Q $, and the response to the query is the depth-$ D $ map $ \xxq $, given its efficient representation.
			
			The one-time work takes $ CN\log N $ operations and $ CN $ storage. The time to answer a query is $ C\log N $. 
		\end{enumerate}
	\end{lemma}

	\begin{proof}
		Repeating the argument of Lemma 3.8 of \cite{JL20-Ext}, we can show that there exists a map
		\begin{equation*}
			\E_Q:\ctp(E) \times \pos \to \ctp((1+c_G)Q)
		\end{equation*}
		such that the following hold.
		\LAQ{eq.EQ-bound}{Given $ (f,M) \in \ctp(E) \times\pos $ with $ \norm{f}_{\ctp(E)} \leq M $, we have
			\begin{enumerate}[(a)]
				\item $ \E_Q(f,M) \geq 0 $ on $ (1+c_G)Q $;
				\item $ \E_Q(f,M)(x) = f(x) $ for $ x \in E \cap (1+c_G)Q $;
				\item $ \norm{\E_Q(f,M)}_{\ct((1+c_G)Q)} \leq CM $; and
				\item $ \jet_{\xqs}\E_Q(f,M) = T_Q(f,M) $, with $ \xqs $ as in Lemma \ref{lem.rep}, $ T_Q $ as in Definition \ref{def.TQ}. 
		\end{enumerate}}
		
		\LAQ{eq.EQ-depth}{For each $ x \in (1+c_G)Q $, there exists a set $ S_Q(x) \subset E $ with $ \#(S_Q(x)) \leq D_0 $ for some universal constant $ D_0 $, such that the following holds: Given $ (f,M),(g,M) \in \ctp(E)\times\pos $ with $ \norm{f}_{\ctp(E)}, \norm{g}_{\ctp(E)} \leq M $ and $ f|_{S_Q(x)} = g|_{S_Q(x)} $, we have $ \jet_x^+ \E_Q(f,M) = \jet_x^+\E_Q(g,M) $.}
		
		To prove Lemma \ref{lem.loc-ext}, we need to dissect the operator $ \E_Q $ and analyze its complexity. 
		
		As in Lemma 3.8 of \cite{JL20-Ext}, the operator $ \E_Q $ takes the following form:
		\begin{equation}
			\begin{split}
				\E_Q(f,M) &:= {T_Q(f,M)} + (1-\psi)\cdot \widetilde{\mathcal{E}}_Q(f,M),\,\text{ where }\\
				\widetilde{\mathcal{E}}_Q(f,M)
				&:= 
				\bigg(
				\overbrace{{ V \circ  
						\underbrace{
							\bigg[
							\brac{  
								\Delta_{f,M}^Q \Ebpm + (1-\Delta_{f,M}^Q)\Eb 
							} 
							\underbrace{
								\brac{(f - {T_Q(f,M)}\big|_{E}) \circ \Phi^{-1}\big|_{\R \times \set{0}} }
							}_{\text{straightening local data}}
							\bigg]
						}_{\text{one-dimensional extension}}   }}^{\text{vertical extension}}
				\bigg) 
				\circ \Phi\,.
			\end{split}
			\label{eq.EQ}
		\end{equation}
		Here, in the order of appearance in \eqref{eq.EQ},
		\begin{itemize}
			\item $ T_Q $ is as in Definition \ref{def.TQ};
			\item $ \psi \in \ctp(\rt) $ with $ \psi \equiv 1 $ near $ \xqs $ (see Lemma \ref{lem.rep}), $ \supp{\psi} \subset B(\xqs,\frac{c_0}{2}\dq) $ with $ c_0 $ as in Lemma \ref{lem.rep}, and $ \abs{\d^\alpha \psi} \leq C\dq^{-\abs{\alpha}} $;
			\item $ V $ is the vertical extension map $ V(g)(t_1,t_2) := g(t_1) $, for $ g $ defined on a subset of $ \R $;
			\item $ \Delta_{f,M}^Q $ is an indicator function defined by
			\begin{equation*}
				\Delta_{f,M}^Q := \begin{cases}
					1 &\text{ if $ T_Q(f,M) $ is not the zero polynomial}\\
					0 &\text{ otherwise}
				\end{cases}\,\text{;}
			\end{equation*}
			\item $ \Eb $ and $ \Ebpm $, respectively, are the one-dimensional extension operators associated with Theorem \ref{thm.bd-1d} and Theorem \ref{thm.bd-1dpm} (see also Theorems 2.A and 2.B of \cite{JL20});
			\item $ \Phi $ is the diffeomorphisms in Lemma \ref{lem.CZ}(B). 
		\end{itemize}

		We bring ourselves back to the setting of Lemma \ref{lem.loc-ext}. Recall the definition of $ \jet_x^+ $ as in \eqref{eq.jet-def}. We want to define the maps $ \set{\Xi_{x,Q} : x \in (1+c_G)Q} $ by
		\begin{equation}
			\Xi_{x,Q} := \jet_x^+\circ\E_Q
			\for x \in (1+c_G)Q.
			\label{eq.xxq-def}
		\end{equation}
		
		Lemma \ref{lem.loc-ext}(A) follows from \eqref{eq.EQ-depth}. Lemma \ref{lem.loc-ext}(B) follows from \eqref{eq.EQ-bound}. 
		
		It remains to examine Lemma \ref{lem.loc-ext}(C). Suppose we have performed the necessary one-time work using at most $ CN\log N $ operations and $ CN $ storage. 
		
		Let $ x \in (1+c_G)Q $ be given. 
		
		\begin{enumerate}[Step 1.]
			\item We compute
			\begin{equation*}
				t_x := Proj_{u_Q^\perp}(x-\erep(Q)).
			\end{equation*}
			Here, 
			\begin{itemize}
				\item $ Proj_{u_Q^\perp} $ denotes orthogonal projection onto $ \R u_Q^\perp $;
				\item the pair $ \set{u_Q, u_Q^\perp} $ is as in Lemma \ref{lem.uQ}; and
				\item $ \erep(Q) $ is as in Lemma \ref{lem.FK-CZ}(C). 
			\end{itemize}
			All the procedures involved in this step require at most $ C\log N $ operations, thanks to Lemma \ref{lem.FK-CZ}(B) and Lemma \ref{lem.uQ}. 
			
			\item Let $ \rho $ be the rotation about the origin specified by $ e_2 \mapsto u_Q $. We can compute $ \jet_x^+\rho $.

			\newcommand{\bst}{\overline{S}(t_x)}
			\newcommand{\ebx}{\overline{E}_Q}
			\item Let $ u_Q^\perp $ and $ Proj_{u_Q^\perp} $ be as in Step 1. We set
			\begin{equation*}
				\ebx := Proj_{u_Q^{\perp}}(E\cap (1+c_G)Q -\erep(Q)) \subset \R.
			\end{equation*}
			Recall from Lemma \ref{lem.sort} that we can compute the {\em sorted } list $ \ebx $ for each $ Q \in \Lsk $ using at most $ CN\log N $ operations and $ CN $ storage. 
			
			Let $ \ctp(\ebx) $ and $ \ct(\ebx) $ be the one-dimensional trace spaces. Note that we have sorted $ \ebx $. Let $ \bxi^{t_x}_+ $ and $ \bxi^{t_x}_\pm $, respectively be the maps associated with $ \ctp(\ebx) $ and $ \ct(\ebx) $, as in Theorems \ref{thm.bd-1d} and \ref{thm.bd-1dpm}.
			
			\item Recall from Lemma \ref{lem.CZ}(B) that the diffeomorphism $ \Phi $ is defined in terms of a function $ \phi $, satisfying \eqref{eq.phi-1} and \eqref{eq.phi-2}. We compute $ \bjet_{t_x}^+\phi $, where $ \bjet_{t_x}^+ $ is the single-variable two-jet at $ t_x $. We can accomplish this by simply setting $ \bjet_{t_x}^+\phi := \bxi_{\pm}^{t_x}(\phi|_{\ebx}) $, with $ \bxi_{\pm}^{t_x} $ as in Theorem \ref{thm.bd-1dpm}. Since we have already sorted the set $ \ebx $ in Step 3, computing $ \bxi_\pm^{t_x}(\phi|_{\ebx}) $ requires at most $ C\log N $ operations.

			\item Similar to Step 4, the query time for $ \bjet_{t_x}^+\circ \Eb(\cdot)$\footnote{Note that $ \Eb(\cdot) $ is only defined for $ \tilde{f}:\ebx \to \pos $} and $ \bjet_{t_x}^+\circ \Ebpm(\cdot) $ is $C\log N$, since set $ \ebx $ has been sorted in Step 3.

			\item By Lemma \ref{lem.CZ}(B), the diffeomorphism $ \Phi = (\Phi_1, \Phi_2) $ and its inverse $ \Phi^{-1} = (\Psi_1, \Psi_2)  $ are given by
			\begin{equation*}
				\begin{split}
					\Phi\circ\rho(t_1, t_2) &= (t_1, t_2 - \phi(t_1)), \text{ and }
					\\
					\rho^{-1}\circ\Phi^{-1}(t_1', t_2') &= (t_1', t_2' + \phi(t_1')).
				\end{split}
			\end{equation*}
			Therefore, we can compute $ \jet_x^+\Phi_i $ and $ \jet_x^+\Psi_i $, $ i = 1,2 $, from the (single-variable) two-jet of $ \phi $. 
			
			\item We compute $ T_Q(f,M) $, as in Definition \ref{def.TQ}. Computing $ \ssq $ as in \eqref{eq.ssq-def} requires at most $ C\log N $ operations, by Lemma \ref{lem.FK-CZ}(C) and Lemma \ref{lem.rep}. After that, we can compute $ T_Q(f,M) $ in $ C $ operations. See Remark \ref{rem.TQ}.

		\end{enumerate}
		
		Combining all the steps above, we see that we can compute the map $ \Xi_{x,Q} $ in \eqref{eq.xxq-def} via formula \eqref{eq.EQ} using at most $ C\log N $ operations. After that, given $(f,M) \in \cte\times\pos$, we can compute $\Xi_{x,Q}(f,M)$ in $C$ operations.
		
		This proves Lemma \ref{lem.loc-ext}.
	\end{proof}

	\subsection{Partitions of unity}\label{sect:pou}

	Recall the definition of $ \jet_x^+ $ as in \eqref{eq.jet-def}.

	We can construct a partition of unity $ \set{\theta_Q : Q \in \Lz} $ that satisfies the following properties:
	\begin{itemize}
		\item $\theta \geq 0 $;
		\item $ \sum_{Q \in \Lz}\theta_Q \equiv 1 $;
		\item $ \supp{\theta_Q}\subset (1+c_G/2)Q $ for each $ Q \in \Lz $;
		\item For each $ Q \in \Lz $, $ \abs{\d^\alpha\theta_Q} \leq C\dq^{2-\abs{\alpha}} $ for $ \abs{\alpha} \leq 2 $;
		\item After one-time work using at most $ CN\log N $ operations and $ CN $ storage, we can answer queries as follows: Given $ x \in \rt $ and $ Q \in \Lz $, we return $ \jet_x^+\theta_Q $. The time to answer query is $ C\log N $.
	\end{itemize}

	See Section 28 of \cite{FK09-Data-2} for details.

	\subsection{Proof of Theorem \ref{thm.bd-alg}}\label{proof-thm.bd-alg}

	\begin{proof}[Proof of Theorem \ref{thm.bd-alg}]
		Slightly modifying the proof of Theorem 1 of \cite{JL20-Ext}, we can show that there exists a map
		\begin{equation}
			\E:\ctp(E)\times\pos \to \ctp(\rt)
		\end{equation}
		such that the following hold.
		\LAQ{eq.E-bound}{Given $ (f,M) \in \ctp(E) \times\pos $ with $ \norm{f}_{\ctp(E)} \leq M $, we have
			\begin{enumerate}[(a)]
				\item $ \E(f,M) \geq 0 $ on $ \rt $;
				\item $ \E(f,M)(x) = f(x) $ for $ x \in E $; and
				\item $ \norm{\E(f,M)}_{\ct(\rt)} \leq CM $.
		\end{enumerate}}
		
		\LAQ{eq.E-depth}{For each $ x \in \rt $, there exists a set $ S(x) \subset E $ with $ \#(S(x)) \leq D $ for some universal constant $ D $, such that the following holds: Given $ (f,M),(g,M) \in \ctp(E)\times\pos $ with $ \norm{f}_{\ctp(E)}, \norm{g}_{\ctp(E)} \leq M $ and $ f|_{S(x)} = g|_{S(x)} $, we have $ \jet_x^+ \E(f,M) = \jet_x^+\E(g,M) $.}
		
		Moreover, $ \E $ takes the form of
		\begin{equation}
			\E(f,M)(x) := \sum_{Q \in \Lz} \theta_Q(x) \cdot \E_Q^\sharp(f,M)(x) = \sum_{Q \in \Lambda(x)} \theta_Q(x)\cdot\E_Q^\sharp(f,M)(x),
		\end{equation}
		where
		\begin{itemize}
			\item $ \set{\theta_Q:Q \in \Lz} $ is the partition of unity constructed in Section \ref{sect:pou}; 
			\item $ \Lambda(x) $ is the set in Lemma \ref{lem.FK-CZ}(A); and
			\item $ \E_Q^\sharp $ is defined by the following rule.
			\begin{itemize}
				\item Suppose $ Q \in \Lsk $. Then $ \E_Q^\sharp(f,M) := \E_Q(f,M) $ with $ \E_Q $ as in Lemma \ref{lem.loc-ext};
				\item Suppose $ Q \in \Ls\setminus \Lsk $. Then $ \E_Q^\sharp := T_w^\xqs \circ T_Q $, with $ T_Q $ as in Definition \ref{def.TQ}, $ \xqs $ as in Lemma \ref{lem.rep}, and $ T_w^\xqs $ as in Lemma \ref{lem.WT}(B) (associated with the singleton $ \{\xqs\} $). 
				\item Suppose $ Q \in \Le $. Then $ \E_Q^\sharp := T_w^{x_{\mu(Q)}^\sharp}\circ T_{\mu(Q)} $, with $ \mu $ as in Lemma \ref{lem.mu}, $ x_{\mu(Q)}^\sharp $ as in Lemma \ref{lem.rep}, $ T_{\mu(Q)} $ as in Definition \ref{def.TQ}, and $ T_w^{x_{\mu(Q)}^\sharp} $ as in Lemma \ref{lem.WT}(B) (associated with the singleton $ \{x_{\mu(Q)}^\sharp\} $). 
				\item Suppose $ Q \in \Lz \setminus (\Ls \cup \Le) $. Then $ \E_Q^\sharp :\equiv 0 $. 
			\end{itemize}
		\end{itemize}
		
		We set
		\begin{equation}
			\Xi_x(f,M) := \jet_x^+\circ\E(f,M) = \sum_{Q \in \Lambda(x)} \jet_x^+\theta_Q\odot_x^+\jet_x^+\circ\E_Q^\sharp(f,M)
			\for x \in \rt.
			\label{eq.xx-def}
		\end{equation}
		
		Theorem \ref{thm.bd-alg}(A) follows from \eqref{eq.E-bound} and Theorem \ref{thm.bd-alg}(B) follows from \eqref{eq.E-depth}. 
		
		We now turn to Theorem \ref{thm.bd-alg}(C). Suppose we have performed the necessary one-time work using at most $ CN\log N $ operations and $ CN $ storage. 
		
		By Lemma \ref{lem.FK-CZ}(A) and Section \ref{sect:pou}, we can compute $ \Lambda(x) $ and $ \set{\jet_x^+\theta_Q: Q \in \Lambda(x)} $ using at most $ C\log N $ operations.
		
		By Lemma \ref{lem.loc-ext}, we can compute 
		\begin{equation*}
			\set{\jet_x^+\circ\E_Q(f,M) : Q \in \Lsk \cap \Lambda(x)} 
		\end{equation*}
		using at most $ C\log N $ operations, after computing $ \Lambda(x) $.
		
		By Lemma \ref{lem.rep} and Remark \ref{lem.TQ}, we can compute
		\begin{equation*}
			\set{\jet^+_x\circ T_w^{\xqs} \circ T_Q(f,M) : Q \in \Lambda(x) \cap (\Ls \setminus \Lsk)}
		\end{equation*}
		using at most $ C\log N $ operations, after computing $ \Lambda(x) $.
		
		By Lemma \ref{lem.mu}, Lemma \ref{lem.rep} and Remark \ref{rem.TQ}, we can compute 
		\begin{equation*}
			\set{\jet_x^+\circ T_w^{x_{\mu(Q)}^\sharp}\circ T_{\mu(Q)}(f,M) : Q\in \Le\cap \Lambda(x)}
		\end{equation*}
		using at most $ C\log N $ operations, after computing $ \Lambda(x) $.
		
		Therefore, we can compute $ \Xi_x $ in \eqref{eq.xx-def} using at most $ C\log N $ operations. Given $(f,M) \in \cte\times\pos$, we can compute $\Xi_x(f,M)$ in $C$ operations. Theorem \ref{thm.bd-alg}(C) follows.
		
		This proves Theorem \ref{thm.bd-alg}.
		
	\end{proof}

	\appendix
	
	\section{Convex quadratic programming problem with affine constraint}
	\label{app:QP}

	Let $ d \geq 0 $ be an integer bounded by a universal constant. We use the standard dot product on $ \R^d $ and $ \R^{2d} $. We use bold-faced letters to denote given quantities. 
	
	\renewcommand{\abs}[1]{|#1|}
	\renewcommand{\A}{\mathbf{A}}
	\newcommand{\B}{\mathbf{B}}
	\renewcommand{\L}{\mathbf{L}}
	
	We consider a general form of the minimization problem \eqref{eq.qp}:
	\begin{equation}
		\text{ Minimize }  \beta^t \A \beta + \sum_{i = 1}^d\abs{\beta^i}
		\quad \text{ subject to } \B \beta = \mathbf{b}.
		\label{eq.QP}
	\end{equation}
	Here, $ \beta = (\beta^1, \cdots, \beta^d)^t \in \R^d $ is the optimization variable, $ \A \in M_{d\times d} $ is a given positive semidefinite matrix, $ \B\in M_{d\times d} $ is a given (singular) matrix, and $ \mathbf{b} $ is a given vector. 
	
	We will solve \eqref{eq.QP} by first augmenting the system \eqref{eq.QP} to remove the absolute values in the objective function. For the augmented system, which is still convex, the solution can be found by solving for a system of linear equalities and inequalities arising from its associated Karush–Kuhn–Tucker (KKT) conditions \cite{BV-CO}. 
	
	We begin with the augmentation. Decomposing $ \beta $ into its positive and negative parts, $ \beta = \beta_+ - \beta_- $, i.e., $ \beta_+^i := \frac{1}{2}(\beta^i + \abs{\beta^i}) $ and $ \beta_-^i := \beta_+^i - \beta^i  $, we arrive at the system:
	\begin{equation}
		\begin{split}
			&\text{ Minimize } 
			\brac{\begin{matrix}
					\beta_+ \\ \beta_-
			\end{matrix}}^t
			\brac{\begin{matrix}
					\A&-\A\\-\A&\A
			\end{matrix}}
			\brac{\begin{matrix}
					\beta_+\\\beta_-
			\end{matrix}}
			+ \brac{\mathbf{1}_{2d}}^t
			\brac{\begin{matrix}
					\beta_+\\\beta_-
			\end{matrix}}
			\\
			&\text{ Subject to }
			\brac{\begin{matrix}
					\B&-\B
			\end{matrix}}	
			\brac{\begin{matrix}
			\beta_+\\\beta_-
			\end{matrix}} = \mathbf{b},\text{ and }
			\brac{\begin{matrix}
					\beta_+\\\beta_-
			\end{matrix}} \geq \mathbf{0}_{2d}.
		\end{split}
		\label{eq.QP-aug}
	\end{equation}
	
	Note that in order for \eqref{eq.QP} and \eqref{eq.QP-aug} to be equivalent, we have to include in \eqref{eq.QP-aug} the additional sign constraint
	\begin{equation}\label{eq.QP-aug-sign-0}
			\beta_+^i \beta_-^i = 0 \for i = 1, \cdots, d;
	\end{equation}
	or equivalently, for some $ I \subset \set{1, \cdots, d} $,
	\begin{equation}\label{eq.QP-aug-sign-1}
		\mathbf{e}_k^t\beta_+ = 0 \for k \in I \text{ and } \mathbf{e}_k^t\beta_- = 0 \for k \in \set{1,\cdots, d}\setminus I.
	\end{equation}
	Here, $ \set{\mathbf{e}_k : k=1,\cdots,d} $ is the standard basis for $ \R^d $.
	
	For convenience, set 
	\begin{equation*}
		\hat{\beta} := \brac{\begin{matrix}
				\beta_+\\\beta_-
		\end{matrix}}
		,\,
		\hat{\A}:= \brac{\begin{matrix}
				\A&-\A\\-\A&\A
		\end{matrix}},\text{ and }
		\hat{\B} := \brac{\begin{matrix}
				\B&-\B
		\end{matrix}} = \brac{\begin{matrix}
		\hat{\B}_1\\\vdots\\\hat{\B}_{j_{\max}}
	\end{matrix}}.
	\end{equation*}
	Let $ \set{\hat{\mathbf{e}}_i : i = 1, \cdots, 2d} $ be the standard basis for $ \R^{2d} $.

	The KKT conditions for \eqref{eq.QP-aug} coupled with \eqref{eq.QP-aug-sign-1} for a {\em fixed} $ I \subset \set{1, \cdots, d} $ are given by
	\begin{equation}\label{eq.KKT}
		\begin{split}
			2\hat{\A}\hat{\beta} - 
			\sum_{i = 1}^{2d}\mu_i\hat{\mathbf{e}}_i 
			+ \sum_{j = 1}^{j_{\max}}\lambda_j \hat{\B}_j^t 
			+ \sum_{k \in I }\nu_{k}\hat{\mathbf{e}}_{k} 
			+ \sum_{k \in \set{1, \cdots, d}\setminus I}\nu_{k}\hat{\textbf{e}}_{k + d}
			&= \mathbf{0}_{2d},
			\\
			\hat{\beta} &\geq \mathbf{0}_{2d},
			\\
			\hat{\B}\hat{\beta} - \mathbf{b} &= \mathbf{0}_{j_{\max}},
			\\
			\hat{\mathbf{e}}_k^t\hat{\beta} &= \mathbf{0} \for k \in I,
			\\
			\hat{\mathbf{e}}_{k+d}^t\hat{\beta} &= \mathbf{0} \for k \in \set{1, \cdots, d}\setminus I,
			\\
			\mu_i &\geq \mathbf{0} \for i = 1, \cdots, 2d.
			\\
			\sum_{i = 1}^{2d}\mu_i (\hat{\textbf{e}}_i^t\hat{\beta}) &= \mathbf{0}.
		\end{split}
	\end{equation}
	In the above, $ \mu_1, \cdots, \mu_{2d} $, $ \lambda_1, \cdots, \lambda_{j_{\max}} $, $ \nu_1, \cdots, \nu_d $ are multipliers, and $ \hat{\beta} $ is the primal optimization variable.  
	
	Since the matrix $ \hat{\A} $ is positive semidefinite, the primal problem in \eqref{eq.QP-aug} is convex. The KKT conditions are necessary and sufficient for the solutions to be primal and dual optimal\cite{BV-CO}. Hence, solving \eqref{eq.QP-aug} coupled with \eqref{eq.QP-aug-sign-1} for a fixed $ I \subset \set{1, \cdots, d} $ amounts to solving a bounded system \eqref{eq.KKT} of linear inequalities. The latter can be achieved, for instance by the simplex method or elimination\cite{BV-CO}. The number of operations involved is at most (doubly) exponential in system size, which is universally bounded. Therefore, we can solve \eqref{eq.QP-aug} coupled with \eqref{eq.QP-aug-sign-1} for a fixed $ I \subset \set{1, \cdots, d} $ using at most $ C $ operations. 
	
	Finally, we can solve \eqref{eq.QP} using at most $ C $ operations by solving \eqref{eq.QP-aug} coupled with \eqref{eq.QP-aug-sign-1} for every $ I \subset \set{1, \cdots, d} $ and compare the minimizers. 
	
	It is very likely that one can solve \eqref{eq.QP} more efficiently with advanced techniques. Here we content ourselves with the elementary exposition above. We refer the readers to \cite{BV-CO} for a more detailed discussion on convex optimization.

	\bibliographystyle{amsplain}
	
	\providecommand{\bysame}{\leavevmode\hbox to3em{\hrulefill}\thinspace}
	\providecommand{\MR}{\relax\ifhmode\unskip\space\fi MR }
	\providecommand{\MRhref}[2]{%
		\href{http://www.ams.org/mathscinet-getitem?mr=#1}{#2}
	}
	\providecommand{\href}[2]{#2}

\end{document}